\documentclass{amsart}

\usepackage[alphabetic,initials,nobysame]{amsrefs}
\usepackage{amssymb,mathtools,mathrsfs,comment,hyperref}
\usepackage[shortlabels]{enumitem}
\usepackage{color}
\usepackage{caption}
\usepackage[percent]{overpic}
\usepackage{tikz}

\theoremstyle{plain}
\newtheorem{theorem}{Theorem}
\newtheorem{lemma}[theorem]{Lemma}

\newtheorem{corollary}[theorem]{Corollary}

\theoremstyle{definition}

\newtheorem{claim}{Claim}

\theoremstyle{remark}
\newtheorem{remark}[theorem]{Remark}
\newtheorem{question}[theorem]{Question}

\numberwithin{equation}{section}
\numberwithin{theorem}{section}
\numberwithin{conjecture}{section}

\newcommand{\br}{\overline}
\newcommand{\R}{\mathbb R}
\newcommand{\T}{\mathbb T}
\newcommand{\C}{\mathbb C}

\newcommand{\Z}{\mathbb Z}
\newcommand{\N}{\mathbb N}

\DeclareMathOperator{\dist}{{\mathrm{dist}}}
\DeclareMathOperator{\diam}{{\mathrm{diam}}}

\DeclareMathOperator{\inter}{{\mathrm{int}}}
\DeclareMathOperator{\re}{{\mathrm{Re}}}
\DeclareMathOperator{\im}{{\mathrm{Im}}}
\DeclareMathOperator{\length}{{\mathrm{length}}}
\DeclareMathOperator{\Area}{{\mathrm{Area}}}

\DeclareMathOperator{\area}{\mathrm{Area}}

\DeclareMathOperator{\dimh}{\dim_{\mathscr H} }
\DeclareMathOperator{\proj}{\mathrm{proj}}

\begin{document}
\title[Dimension of packings by smooth curves]{On the Hausdorff dimension of the residual set of a packing by smooth curves}

\author{Steven Maio}
\author{Dimitrios Ntalampekos}
\address{Institute for Mathematical Sciences, Stony Brook University, Stony Brook, NY 11794, USA.}

\thanks{The second author is partially supported by NSF Grant DMS-2000096.}
\email{steven.maio@stonybrook.edu} 
\email[Corresponding author]{dimitrios.ntalampekos@stonybrook.edu}

\date{\today}
\keywords{Packing, circle packing, residual set, Hausdorff dimension, convex curve, chord-arc, Ahlfors regular}
\subjclass[2010]{Primary 52C15, 28A78; Secondary 52A10, 28A80.}

\begin{abstract}
Let a planar residual set be a set obtained by removing countably many disjoint topological disks from an open set in the plane. We prove that the residual set of a planar packing by curves that satisfy a certain lower curvature bound has Hausdorff dimension bounded away from $1$, quantitatively, depending only on the curvature bound. As a corollary, the residual set of any circle packing has Hausdorff dimension uniformly bounded away from $1$. This result generalizes the result of Larman, who obtained the same conclusion for circle packings inside a square. We also show that our theorem is optimal and does not hold in general without lower curvature bounds. In particular, we construct packings by strictly convex, smooth curves whose residual sets have dimension $1$. On the other hand, we prove that any packing by strictly convex curves cannot have $\sigma$-finite Hausdorff $1$-measure.
\end{abstract}

\maketitle

\section{Introduction}

Let $\Omega\subset \R^2$ be an open set. A collection $\mathcal P_{\Omega}=\{D_i\}_{i\in \N}$ of open Jordan regions in $\R^2$ is called a \textit{packing inside $\Omega$} if $D_i\subset \Omega$ for all $i\in \N$ and $D_i\cap D_j=\emptyset$ for all $i,j\in \N$ with $i\neq j$. Note that the closures of two Jordan regions $D_i,D_j$, $i,j\in \N$, might intersect. The \textit{residual set of a packing $\mathcal P_{\Omega}= \{D_i\}_{i\in \N}$} is defined to be the locally compact set
\begin{align*}
\mathcal S= \Omega\setminus \bigcup_{i\in \N}D_i.
\end{align*}
We refer to $\mathcal P_{\Omega}$ as a \textit{packing by the regions $D_i$} or equivalently \textit{by the curves} $\partial D_i$. Moreover, we refer to packings by disks as {circle} packings or disk packings.

We first give some background on the estimation of the Hausdorff dimension of the residual sets of various types of packings. Eggleston in \cite{Eggleston:gasket} studied the Hausdorff dimension of packings inside an equilateral triangle $\Omega$ by equilateral oppositely oriented triangles. He proved that the minimal Hausdorff dimension is $\log_2 3$, which is attained by the Sierpi\'nski gasket; see Figure \ref{figure:gasket}. Then Hirst in \cite{Hirst:packing} proved that the Hausdorff dimension of the Apollonian gasket (Figure \ref{figure:apollonian}) is strictly between $1$ and $2$. Later, using the techniqies of Eggleston, Larman \cite{Larman:DimensionPackings} proved that any disk packing inside a square has Hausdorff dimension bounded below by $1.03$. Furthermore, in \cite{Larman:spheres} he used the $2$-dimensional result to prove inductively that a packing inside an $n$-dimensional cube by $n$-balls has always dimension strictly larger than $n-1$.  

A closely related topic is the study of the \textit{exponent} of disk packings, which is another relevant notion of dimension that is at least as large as the Hausdorff dimension by a result of Larman \cite{Larman:packing}. We direct the reader to \cite{Melzak:packing,Wilker:packing2, Foster:packing,Boyd:packing,Wilker:packing} for a sequence of results in the subject. Many of these results aim at estimating the Hausdorff dimension of the Apollonian gasket. Boyd \cite{Boyd:apollonian_bound} proved that the the dimension of the Apollonian gasket is between $1.300$ and $1.315$ and more recently Thomas--Dhar \cite{ThomasDhar:gasket}, McMullen \cite{McMullen:gasket}, and Bai--Finch \cite{BaiFinch:gasket} gave an accurate numerical estimate of $1.30568$. It still remains an open problem, posed by Melzak \cite{Melzak:packing}, whether the Apollonian gasket attains the minimum Hausdorff dimension among all disk packings.

\begin{figure}
\centering
\begin{minipage}{.52\textwidth}
\centering
\captionsetup{width=.8\linewidth}
\includegraphics[width=1.\linewidth]{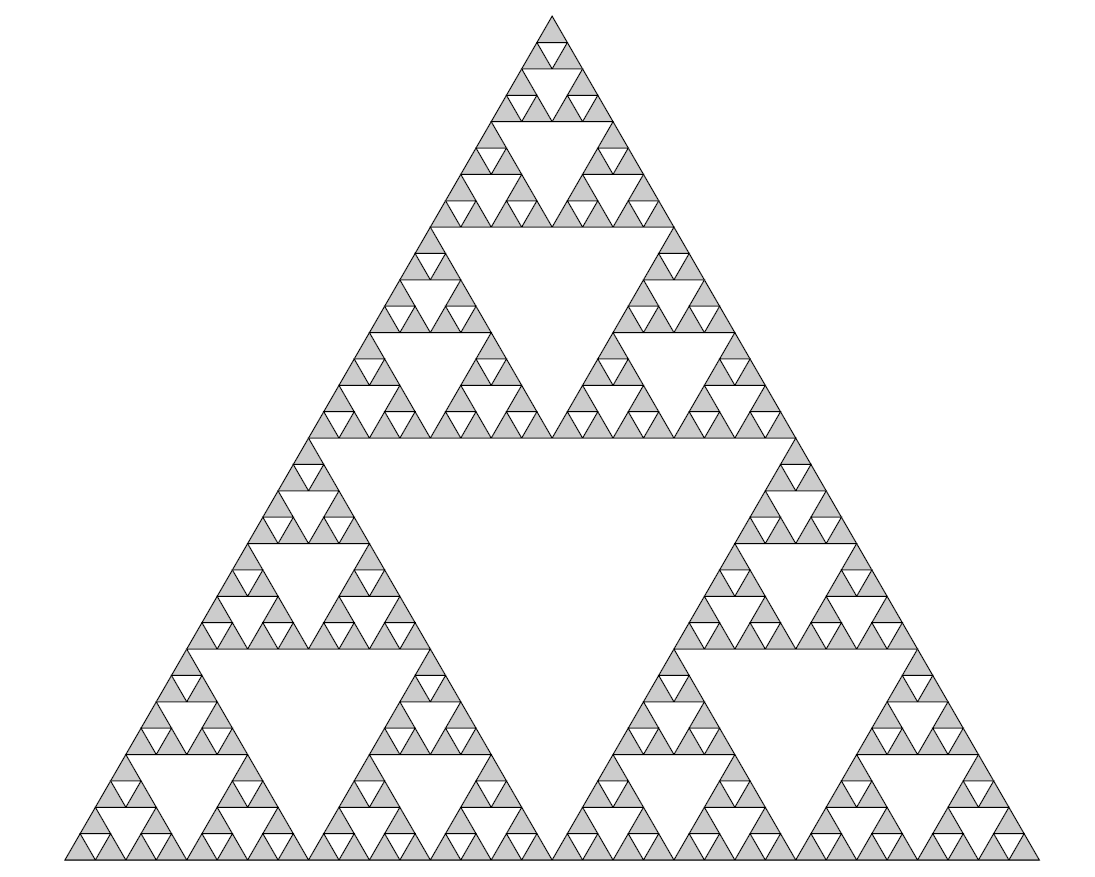}
\captionof{figure}{The Sierpi\'nski gasket.}\label{figure:gasket}
\end{minipage}
\begin{minipage}{.47\textwidth}
\centering
\captionsetup{width=.8\linewidth}
\includegraphics[width=1.\linewidth]{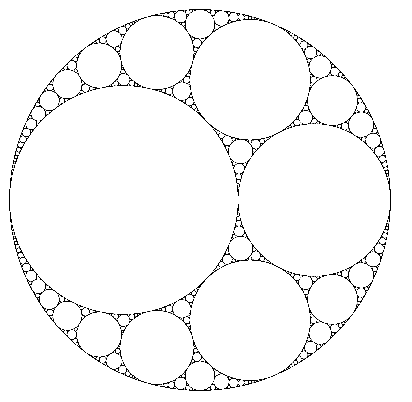}
\captionof{figure}{The Apollonian gasket.}\label{figure:apollonian}
\end{minipage}
\end{figure}

In this work we study lower bounds for the Hausdorff dimension of the residual sets of packings by arbitrary smooth, convex curves, rather than circles. Our main result is the following theorem.

\begin{theorem}\label{theorem:main}
For each $k>0$, there exists a constant $K=K(k)>1$ such that the following holds. Let $\Omega\subset \R^2$ be an open set and let $\mathcal P_{\Omega}=\{D_i\}_{i\in \N}$ be a packing such that for each $i\in \N$ the curve $\partial D_i$ is smooth and has curvature bounded below by $k/\length(\partial D_i)$. Then, the residual set $\mathcal S$ of $\mathcal P_{\Omega}$ satisfies
\begin{align*}
\dim_{\mathscr H} (\mathcal S\cap U) >K>1
\end{align*}
for every open set $U$ with $\mathcal S \cap U \neq \emptyset$. In particular, 
\begin{align*}
\dim_{\mathscr H} (\mathcal S) >K>1.
\end{align*}
\end{theorem}

As already mentioned, Larman in \cite{Larman:DimensionPackings} proves that the residual set of any circle packing \textit{inside a square} has Hausdorff dimension bounded away from $1$. We remark that our result is also new for arbitrary circle packings, not necessarily contained in a  square. It turns out that this does not follow immediately from Larman's result and one needs to refine carefully his argument to achieve this. Hence, we obtain the following corollary.

\begin{corollary}
There exists a constant $K>1$ such that for every open set $\Omega\subset \R^2$ and for every circle packing $\mathcal P$ inside $\Omega$, the residual set $\mathcal S$ of $\mathcal P$  satisfies 
\begin{align*}
\dim_{\mathscr H} (\mathcal S) >K>1.
\end{align*}
\end{corollary}

The locality of our result and the fact that we are not working with packings in a square or rectangle, but with packings in an arbitrary set $\Omega$, pose some considerable complication to the proof.  On the other hand, our proof provides a simplification of Larman's argument; he deals with three cases, two of which are lengthy. Instead, we manage to combine two of Larman's cases into a single case with a short proof. We have attempted to refine the techniques of Eggleston \cite{Eggleston:gasket} and Larman \cite{Larman:DimensionPackings} in order to give a transparent proof of the main theorem, based on general properties of convex, smooth curves, as opposed to the work of Larman, which relies on some special trigonometric identities and relations between chords and tangents of the circle.

In our proof we use some modern tools from analysis in metric spaces, such as chord-arc curves and Ahlfors $2$-regular regions. Roughly speaking, chord-arc curves behave like circles in terms of length, in the sense that the length of arcs connecting two points is not much longer than the length of the chord that connects the same points. Analogously, Ahlfors $2$-regular regions behave like round disks in terms of area. The first step of our proof, which is carried out in Section \ref{section:preliminaries}, is to show that curves $C$ satisfying the curvature bound $\kappa\geq k/\length(C)$ are chord-arc curves and the regions that they bound are Ahlfors $2$-regular, with uniform constants, depending only on $k$. Achieving the uniformity of the constants throughout the paper turns out to be a challenging task.

In Section \ref{section:proof_main} we prove the main theorem. After some reductions, we reduce the statement to an iterative estimate for finite packings; see Section \ref{section:main_estimate}. Then the rest of Section \ref{section:proof_main} is devoted to the establishment of this estimate.

In Theorem \ref{theorem:main} the curves $\partial D_i$ are assumed to be smooth. In fact, we only need $C^2$-smoothness, so that the curvature is defined. Moreover, one can relax even further this assumption to curves that can be approximated by smooth curves. This allows the curves $\partial D_i$ to have corners.

\begin{theorem}\label{theorem:relaxation}
Let $k>0$. The conclusion of Theorem \ref{theorem:main} is true for packings $\mathcal P_{\Omega}=\{D_i\}_{i\in \N}$ for which every curve $\partial D_i$, $i\in \N$, can be approximated in the Hausdorff metric by smooth curves $C$ whose curvature is bounded below by $k/\length(C)$.
\end{theorem}
We discuss the definition of Hausdorff convergence and the proof of this generalization in Section \ref{section:relaxation}. Essentially it follows from a small modification of the auxiliary results that were used in the proof of Theorem \ref{theorem:main}.

If we do not impose bounds on the curvature, then the conclusions of the main theorem do not hold.
\begin{theorem}\label{theorem:counterexample}
Let $\Omega\subset \R^2$ be an open set.
\begin{enumerate}[\upshape(i)]
\item There exists a packing $\mathcal P_{\Omega}$ by convex, smooth curves whose residual set  has $\sigma$-finite Hausdorff $1$-measure and, in particular, Hausdorff dimension equal to $1$.
\item There exists a packing $\mathcal P_{\Omega}$ by strictly convex, smooth curves whose residual set  has Hausdorff dimension equal to $1$.
\end{enumerate}
\end{theorem}
This theorem is proved in Section \ref{section:example}. We remark that in the second example the residual set $\mathcal S$ cannot have $\sigma$-finite Hausdorff $1$-measure and one can only guarantee that $\dim_{\mathscr H}(\mathcal S)=1$. This follows from the next theorem and the observation that any two strictly convex Jordan curves that bound disjoint regions can intersect in at most one point.
\begin{theorem}\label{theorem:dimension_one}
Let $\Omega\subset\R^2$ be an open set and $\mathcal P_{\Omega}=\{D_i\}_{i\in \N}$ be a packing such that $\partial D_i\cap \partial D_j$ is at most countable for each $i,j\in \N$ with $i\neq j$. Then the residual set of $\mathcal P_{\Omega}$ cannot have $\sigma$-finite Hausdorff $1$-measure.
\end{theorem}
The proof of this theorem is elementary and is given in Section \ref{section:example}.

We finish the Introduction with some discussion on the dimension of packings by non-smooth or fractal curves.  Theorem \ref{theorem:counterexample} already shows that one cannot expect any dimension bounds without further assumptions. For example, squares can always be packed perfectly with no gaps, and any such packing has dimension $1$. Instead, we consider a separation condition that prevents the regions in a packing from touching. For $M>0$, we say that two bounded regions $D_1,D_2$ are $M$-relatively separated if 
\begin{align*}
\Delta(D_1,D_2) \coloneqq \frac{\dist(D_1,D_2)}{\min \{\diam(D_1),\diam(D_2)\}} \geq M.
\end{align*}
Moreover, we say that a Jordan curve $C$ is an $M$-quasicircle if for any two points $z,w\in C$ there exists an arc $A$ of $C$ connecting $z$ and $w$ such that $\diam(A)\leq M|z-w|$.  It is proved in \cite[Proposition 4.1]{Ntalampekos:Semihyperbolic} that the residual set of a planar packing $\mathcal P$ by Jordan regions $\{D_i\}_{i\in \N}$ that are pairwise $M$-relatively separated and whose boundaries are $M$-quasicircles has Hausdorff dimension bigger than a constant $K>1$, where $K$ depends only on $M$. The latter dependence follows from a result of Mackay \cite[Theorem 1.1]{Mackay:ConformalDimension}. Such packings appear naturally in the setting of Complex Dynamics (Figure \ref{figure:julia}); for example, see \cite[Theorem 1.10]{BonkLyubichMerenkov:carpetJulia}. 

\begin{figure}
\centering
\includegraphics[scale=0.6]{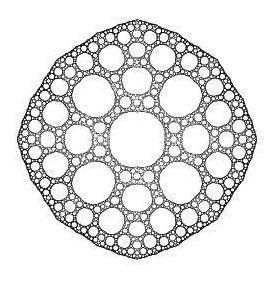}
\caption{The Julia set of $z^2-\frac{1}{16z^2}$.}\label{figure:julia}
\end{figure}

We pose some questions for further study.

\begin{question}
Suppose that there exists a convex, smooth Jordan curve $C$ (with possibly vanishing curvature) such that all of the curves $\partial D_i$ in the packing $\mathcal P_{\Omega}=\{D_i\}_{i\in \N}$ are images of $C$ under Euclidean similarities. Is the Hausdorff dimension of the residual set of $\mathcal P_{\Omega}$ larger than $1$?
\end{question}
If the curvature of $C$ is non-vanishing then this follows immediately from Theorem \ref{theorem:main}. However, if the curvature of $C$ vanishes, then we cannot expect that the local Hausdorff dimension is larger than $1$. For example, suppose that two regions of a packing are smoothened squares with disjoint interiors that meet at a segment in their boundary.  Then the local Hausdorff dimension of the residual set at points of that segment is equal to $1$. Hence, we can only ask whether the global Hausdorff dimension of the residual set is larger than $1$.

\begin{question}
Can the conclusion of Theorems \ref{theorem:main} and \ref{theorem:relaxation} be generalized to packings in higher dimensions by convex, smooth bodies under some curvature bounds? What notion of curvature should one consider?
\end{question}

\bigskip

\section{Preliminaries on curves of bounded curvature}\label{section:preliminaries}
In this section we include some definitions and some preliminaries on curves of bounded curvature. We first discuss some notation. Throughout the paper, we will use symbols $c,c',c'',\dots$ and $c_0,c_1,\dots$ for positive constants. We will mention the dependence of constants on various parameters when necessary, or say that the constants are uniform if there is no dependence. We will freely use the same symbols to denote possibly different constants even within the same proof, whenever this does not lead to a confusion. 

The Euclidean plane $\R^2$ is identified with the complex plane $\C$ and we use interchangeably the notations $(x,y)$ and $x+iy$ for the same point in the plane. We use the notation $|E|$ for the Lebesgue measure of a measurable set $E\subset \R$.

For $s>0$ the \textit{$s$-dimensional Hausdorff measure} $\mathscr H^s(E)$ of a set $E\subset \C$ is defined by
$$\mathscr{H}^{s}(E)=\lim_{\delta \to 0} \mathscr{H}_\delta^{s}(E)=\sup_{\delta>0} \mathscr{H}_\delta^{s}(E),$$
where
$$
\mathscr{H}_\delta^{s}(E)\coloneqq \inf \left\{ \sum_{j=1}^\infty \operatorname{diam}(U_j)^{s}: E \subset \bigcup_j U_j,\, \operatorname{diam}(U_j)<\delta \right\}.
$$
The quantity $\mathscr{H}_\delta^{s}(E)$ is called the \textit{$s$-dimensional Hausdorff $\delta$-content} of $E$. The \textit{Hausdorff dimension}  of $E$ is defined by
\begin{align*}
\dim_{\mathscr H}(E)= \inf\{ s\geq 0 : \mathscr H^s(E)=0 \}= \sup\{ s\geq 0: \mathscr H^s(E)=\infty\}.
\end{align*}
The Hausdorff dimension of a set does not change if in the definition of $\mathscr{H}_\delta^{s}(E)$ we use only covers by closed squares with sides parallel to the coordinate axes, instead of using covers by arbitrary sets $U_j$. 

A \textit{Jordan curve} $C$ is a compact subset of $\C$ that is the homeomorphic image of $\mathbb T=\R/\Z$. The latter is identified with the interval $[0,1]$ after identifying its endpoints. A Jordan curve $C$ is \textit{convex} if the line segment between any two points of $C$ is contained in the closed region that the curve $C$ bounds. We say that $C$ is \textit{strictly} convex if it is convex and does not contain any line segments. A Jordan curve $C$ is $C^2$-\textit{smooth} if there exists a parametrization $\gamma\colon \mathbb T\to C $ that is two times continuously differentiable such that $\dot \gamma(t)\neq 0$ for all $t\in \mathbb T$. If $C$ is a $C^2$-smooth Jordan curve, then we can parametrize it by arc-length. In particular, there exists a constant speed parametrization $\gamma \colon \mathbb T\to C$ such that $|\dot\gamma(t)|=\length(C)$ for all $t\in \T$. Using this specific parametrization, the (unsigned) curvature of the curve $C$ is defined by 
\begin{align*}
\kappa(t)=\length(C)^{-2}\cdot |\ddot \gamma(t)|
\end{align*}
for $t\in \T$. In particular, if $\length(C)=1$, we have $\kappa = |\ddot \gamma|$.	 We note that if we scale a curve by a constant $\lambda>0$, then the curvature scales by $\lambda^{-1}$. If $\kappa \neq 0$ in $\T$, then it follows that $C$ is strictly convex. In what follows, a smooth Jordan curve refers always to a $C^2$-smooth Jordan curve. 

\begin{lemma}\label{lemma:intersection_of_two}
Let $D_1$ and $D_2$ be disjoint Jordan regions such that $\partial D_1$ and $\partial D_2$ are strictly convex curves. Then $\partial D_1 \cap \partial D_2$ contains at most one point.  
\end{lemma}
\begin{proof}
The convexity of $\partial D_1$ and $\partial D_2$ implies that the intersection $\partial D_1\cap \partial D_2$ is connected. The strict convexity implies that $\partial D_1\cap \partial D_2$ cannot contain more than one point. 
\end{proof}

\bigskip

\subsection{Chord-arc curves}
A Jordan curve $C$ is a \textit{chord-arc curve} if there exists a constant $L\geq 1$ such that for any pair of points $a,b\in C$, the shortest arc of $C$ that connects them, denoted by $C|[a,b]$, satisfies
\begin{align*}
\length(C|[a,b])\leq L |a-b|.
\end{align*}
In this case we say that $C$ is an $L$-chord-arc curve. This is a \textit{scale invariant} property. That is, if $C$ is an $L$-chord-arc curve, then for any $r>0$ the curve $rC=\{rz: z\in C\}$ is also $L$-chord-arc.

We will show that if $C$ is a smooth curve whose curvature is bounded below by $k/\length(C)$, as in the assumptions of Theorem \ref{theorem:main}, then $C$ is a chord-arc curve, \textit{quantitatively}, i.e., with constant depending only on $k$. We first establish some auxiliary results.

\begin{lemma}\label{lemma:arc_tangent}
Let $k>0$ and $C$ be a smooth Jordan curve with curvature $\kappa \geq k/\length(C)$. Then the following statements are true.
\begin{enumerate}[\upshape(i)]
\item  For each $z_1,z_2\in C$, if $L(z_2)$ is the line tangent to $C$ at $z_2$, we have 
\begin{align*}
\dist(z_1,L(z_2))\geq \frac{k}{2\pi} \cdot \frac{ \length(C|[z_1,z_2])^2}{\length(C)}. 
\end{align*}
\item The length of the projection of $C$ to any line is at least $\frac{k}{8\pi}\length(C)$.
\end{enumerate}
\end{lemma}

\begin{proof}
We may scale the curve $C$ so that it has unit length and curvature bounded below by $k$.

First we prove (i). Consider a rectangle $R$ that circumscribes $C$ so that one of its sides is contained in the tangent line $L(z_2)$.  The boundary of the rectangle $R$ intersects $C$ in four points and splits $C$ into four components. We first assume that $z_1\neq z_2$ and $z_1,z_2$ lie in the closure of the \textit{same component} of $C\setminus \partial R$, which we denote by $C_1$. Then we claim that
\begin{align}\label{lemma:arc_tangent_adjacent}
\dist(z_1,L(z_2))\geq \frac{k}{\pi} \length(C_1|[z_1,z_2])^2. 
\end{align}

By rotating the curve $C$, we may assume that $L(z_2)$ is a horizontal tangent line and $C$ lies below $L(z_2)$. Hence, the rectangle $R$ has sides parallel to the coordinate axes  and the top side of $R$ is contained in $L(z_2)$. Moreover, without loss of generality we may assume that $z_1$ lies in the closure of the top right component of $C\setminus \partial R$.

We consider a unit speed orientation-preserving parametrization $\gamma\colon \T \to C$. 
The component $C_1$ is parametrized by $\gamma|{[a,b]}$ so that $\dot \gamma(a)=e^{i\pi/2}$ and $\dot \gamma(b)=e^{i\pi}$, where $a,b\in [0,1]$ and $a<b$ (after possibly reparametrizing by a translation). Note that $\gamma(b)=z_2$. We write $\dot \gamma(t)=e^{i\theta(t)}$, $t\in [a,b]$, where $\theta$ is a strictly increasing function, by the strict convexity of $C$, with 
$$\pi/2=\theta(a)\leq \theta(t)\leq \theta(b)=\pi.$$
Moreover, $\ddot \gamma(t)=e^{i\theta(t)} i\dot \theta(t)$, where $\dot\theta(t)\geq 0$ and by assumption we have $\dot \theta (t) \geq k$. 

We set $t_2=b$ and consider $t_1 \in [a,b)$ such that $\gamma(t_1)=z_1$. We have 
\begin{align*}
\dist(z_1,L(z_2)) = \im (z_2-z_1)= \int_{t_1}^{t_2} \sin\theta(t)dt = \int_{t_1}^{t_2} \sin(\theta(t_2)-\theta(t)) dt, 
\end{align*}
since $\theta(t_2)=\pi$. Note that $0\leq \theta(t_2)-\theta(t)\leq \pi/2$ for $t\in [t_1,t_2]$, so 
$$\sin(\theta(t_2)-\theta(t))\geq \frac{2}{\pi} (\theta(t_2)-\theta(t)) =\frac{2}{\pi} \int_{t}^{t_2} \dot \theta \geq \frac{2k}{\pi} (t_2-t).$$
Therefore,
\begin{align*}
\dist(z_1,L(z_2))\geq \frac{2k}{\pi} \int_{t_1}^{t_2} (t_2-t)dt= \frac{k}{\pi} (t_2-t_1)^2.
\end{align*}
Since $C_1$ is parametrized by arc-length we have $t_2-t_1= \length(C_1|[z_1,z_2])$ and this completes the proof of \eqref{lemma:arc_tangent_adjacent}.

Next, we treat the general case. Again, we consider a rectangle $R$ that circumscribes $C$ so that one of its sides, say the top side, is contained in $L(z_2)$.  Now, suppose that $z_1$ lies in one of the bottom components, say the bottom right one. We denote by $z_3$ the rightmost point of $C$, which is a point of intersection with $\partial R$. Then $$\dist(z_1,L(z_2))= \dist(z_3,L(z_2))+\im (z_3-z_1).$$ By \eqref{lemma:arc_tangent_adjacent}, we have $\dist(z_3, L(z_2)) \geq \frac{k}{\pi}\length(C|[z_3,z_2])^2$.  
Note that the $\im(z_3-z_1)$ is \textit{not} the distance of $z_1$ or $z_3$ to a tangent at $z_3$ or $z_1$, respectively, so we cannot apply \eqref{lemma:arc_tangent_adjacent} directly here. 

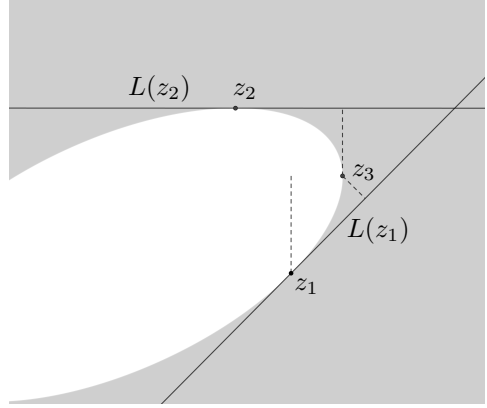
\begin{figure}
	\begin{overpic}[width=.5\linewidth]{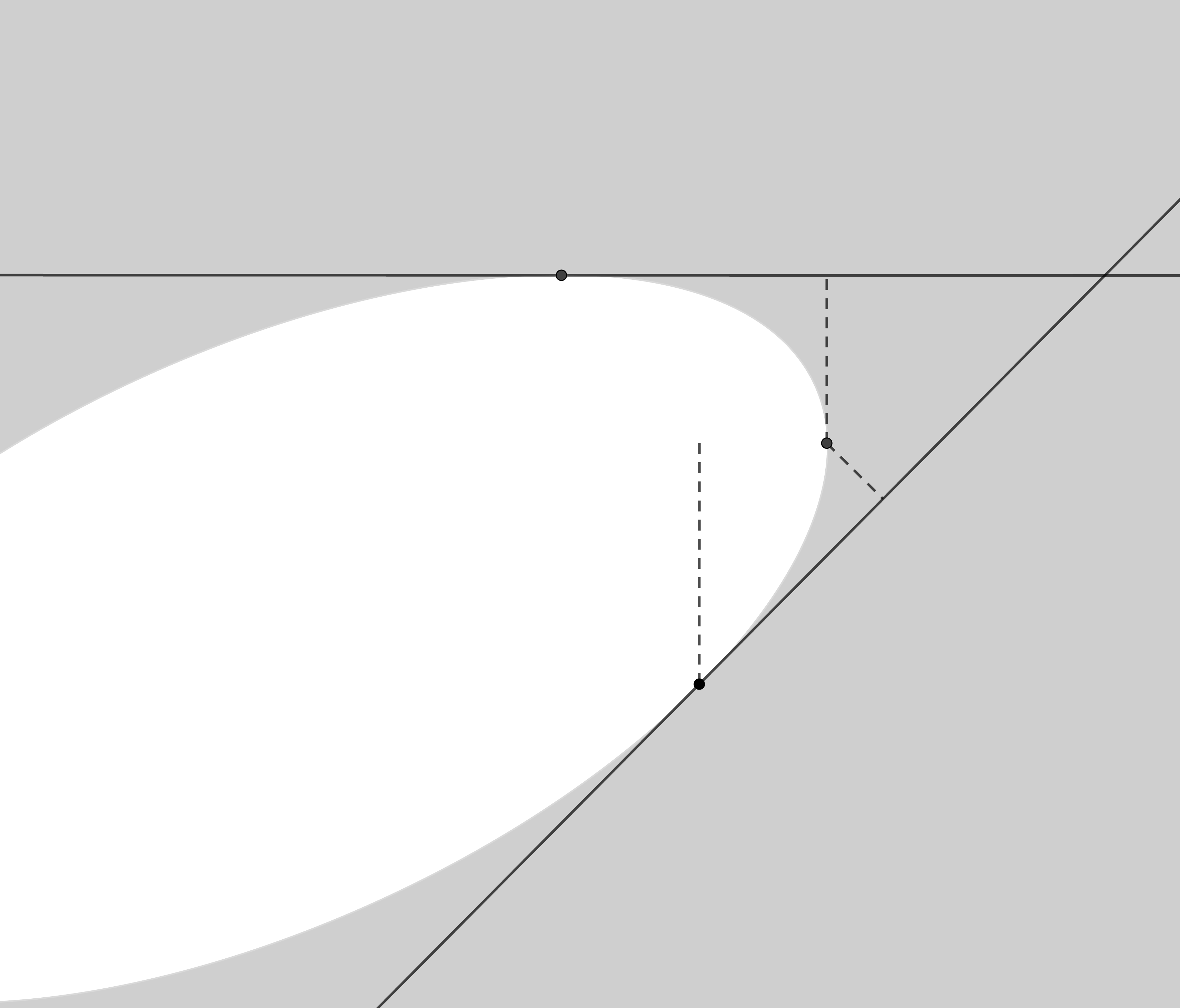}
		\put (47, 65) {$z_2$}
		\put (60, 24) {$z_1$}
		\put (72, 48) {$z_3$}
		\put (25, 65) {$L(z_2)$}
		\put (71, 35) {$L(z_1)$}
	\end{overpic}
	\caption{The distances $\im(z_3-z_1)$, $\dist(z_3,L(z_1))$, and $\dist(z_3,L(z_2))$, when $z_1,z_2$ do not lie in the closure of the same component of $C\setminus \partial R$.}\label{figure:distances}
\end{figure}

Consider the tangent line $L(z_1)$ at $z_1$. Then 
$$\im(z_3-z_1) \geq \dist( z_3, L(z_1)).$$
See Figure \ref{figure:distances}. This follows from the general fact that if $\zeta_3$ is any point in the first quadrant of the plane, then its distance to the line $y=\lambda x$, $ 0\leq \lambda\leq \arg(\zeta_3)$, is maximized when $\lambda=0$, in which case, the maximum distance is $\im(\zeta_3)$. Now consider a rectangle $R_1$ circumscribing $C$ so that one of the sides of $R_1$ is contained in $L(z_1)$. Then $z_1$ and $z_3$ lie in the closure of the same component of $C\setminus \partial R_1$. Hence, by applying \eqref{lemma:arc_tangent_adjacent}, we have 
$$\dist(z_3,L(z_1)) \geq \frac{k}{\pi}\length(C|[z_1,z_3])^2.$$
Summarizing, we have
\begin{align*}
\dist(z_1,L(z_2)) &\geq  \frac{k}{\pi} (\length(C|[z_3,z_2])^2 +\length(C|[z_1,z_3])^2)\\
&\geq  \frac{k}{2\pi} (\length(C|[z_3,z_2]) +\length(C|[z_1,z_3]))^2 \\
&\geq  \frac{k}{2\pi} (\length(C|[z_1,z_2])^2.
\end{align*}
This completes the proof of (i). 

For (ii), consider a rectangle $R$ that circumscribes $C$ and without loss of generality we suppose that its sides are parallel to the coordinate axes. Let the length of the horizontal sides be $a$ and of the vertical sides be $b$. It suffices to show that  $\min\{a,b\}\geq \frac{k}{8\pi}$.  If we remove from $R$ the closed region bounded by $C$, then we obtain four components $R_i$, $i\in \{1,\dots,4\}$, such that the boundary of $R_i$ consists of a horizontal segment $A_i$, a vertical segment $B_i$ and an arc $C_i\subset C$ to which $A_i$ and $B_i$ are tangent. By applying \eqref{lemma:arc_tangent_adjacent} twice, we have
\begin{align*}
\length(A_i) \geq  \frac{k}{\pi} \length(C_i)^2 \quad \textrm{and}\quad \length(B_i) \geq  \frac{k}{\pi} \length(C_i)^2.
\end{align*}
Combining these, we have 
\begin{align*}
2a &=\sum_{i=1}^4 \length(A_i) \geq \frac{k}{\pi} \sum_{i=1}^4\length(C_i)^2 
\geq \frac{k}{4\pi} \left(\sum_{i=1}^4 \length(C_i)\right)^2 = \frac{k}{4\pi},
\end{align*}
since $\length(C)=1$. Therefore, $a\geq \frac{k}{8\pi}$. The same estimate holds for $b$.
\end{proof}

\begin{theorem}\label{theorem:chordarc}
For each $k>0$ there exists $L>0$ such that if $C$ is a smooth Jordan curve  with curvature $\kappa\geq k/\length(C)$, then $C$ is an $L$-chord-arc curve. 
\end{theorem}

\begin{remark}\label{remark:lune}
The conclusion of the theorem does not hold if we merely assume that $\kappa\geq k$, without including the length term. As an example, consider two circles $C_1,C_2$ of length $2\pi$ and curvature $1$ that intersect. Let $C$ be the curve bounding the convex lune formed by the intersection of $C_1$ and $C_2$. One can smoothen $C$ so that it has no corners and curvature bounded below by $1$. However, if $C_1$ and $C_2$ tend to be (externally) tangent to each other, forcing $C$ to have small length, then $C$ cannot be a chord-arc curve with a uniform constant; see Figure \ref{figure:lune}.
\end{remark}

\begin{figure}
	\begin{overpic}[width=.75\linewidth]{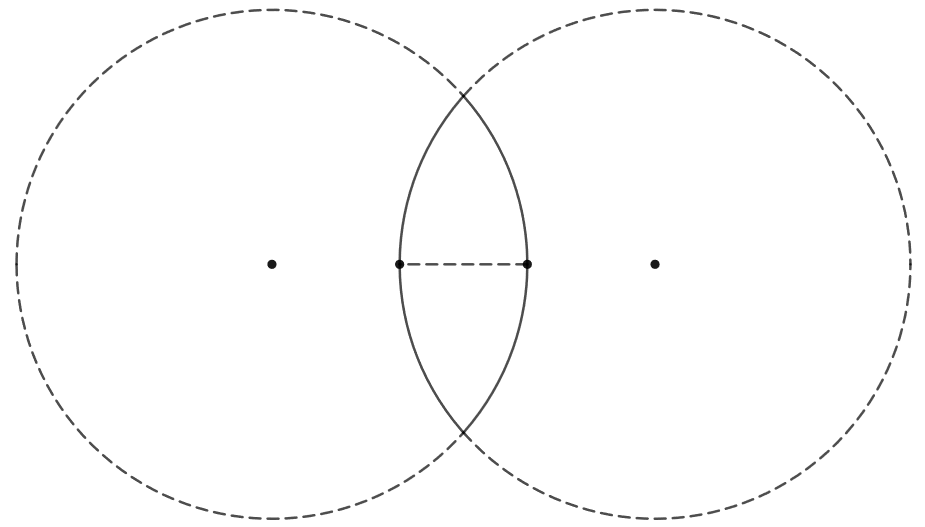}
		\put (30,30) {$(0,0)$}
		\put (72,30) {$(c,0)$}
		\put (58,26) {$p$}
		\put (44,26) {$q$}
	\end{overpic}
	\caption{Two circles $C_1, C_2$ of radius $1$ and the lune $C$ that results from their intersection. We have $p=(1,0)$,  $q=(c-1,0)$, and $|p-q|=2-c$. The arc-length between the two points is $2\arccos(c/2)$. Since $2\arccos(c/2)/(2-c)\to\infty$ as $c\to2^-$, there is no uniform chord-arc constant for this family of curves, although the curvature is bounded below.}\label{figure:lune}
\end{figure}
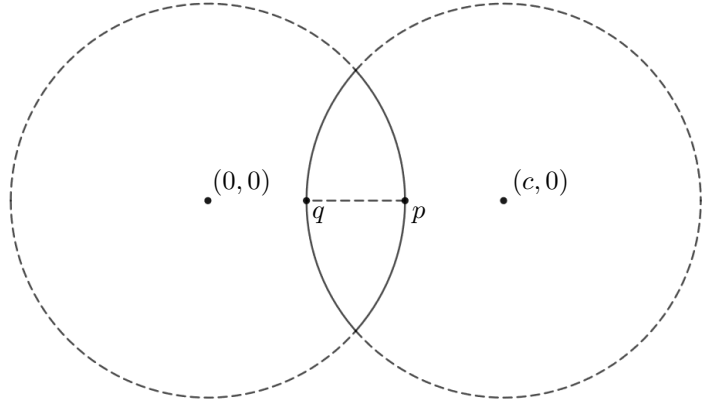

\begin{proof}
It suffices to prove the statement under the assumption that $C$ has unit length and curvature bounded below by $k$.

Suppose that $\gamma$ is a unit speed orientation-preserving parametrization of $C$. Let $z_1,z_2\in C$, $z_1\neq z_2$, and $t_1,t_2\in \T$ be such that $\gamma(t_1)=z_1$ and $\gamma(t_2)=z_2$. By reversing the roles of $t_1$ and $t_2$ if necessary, we may write  $\dot \gamma(t_1)= e^{i\theta_1}$ and $\dot \gamma(t_2)=e^{i\theta_2}$, where $0< \theta_2-\theta_1\leq \pi$. We reparametrize $\gamma$ by composing it with a translation of $\T$ so that $t_1=0$ and $t_2\in[0,1]$ with $t_1<t_2$. By the strict convexity of $C$, there exists a strictly increasing function $\theta(t)$, $0\leq t\leq 1$, such that $\dot \gamma(t)= e^{i\theta(t)}$ for $0\leq t\leq 1$ and $\theta(t_i)=\theta_i$, $i=1,2$. Note that $\dot\theta \geq k$, since $|\ddot \gamma|\geq k$. Finally, by rotating the curve $\gamma$ with a rigid motion of the plane, we may have that $\theta(t_1)=-\alpha$ and $\theta(t_2)=\alpha$, where $0< \alpha\leq \pi/2$.

We have
\begin{align*}
\left|\int_{t_1}^{t_2} e^{i\theta(s)}\, ds \right| &\geq \int_{t_1}^{t_2} \cos\theta(s) \, ds \geq (t_2-t_1) \cos \alpha \geq \length(C|[z_1,z_2]) \cos\alpha.
\end{align*}
Hence, if $\pi/2 -\alpha\geq \delta$, where $\delta\in (0,\pi/4]$ is to be chosen, then 
\begin{align*}
|z_1-z_2|\geq  \length(C|[z_1,z_2]) \cos(\pi/2-\delta).
\end{align*}

Now, suppose $\pi/2-\alpha<\delta$. We consider the point $z_3=\gamma(t_3)\in C$, where $t_3\in (0,1)$ and $\theta(t_3)=\theta(t_1)+\pi= \pi-\alpha$. We have
\begin{align*}
|z_1-z_2|\geq |z_3-z_1|-|z_3-z_2|. 
\end{align*}
The  tangent lines at $z_1$ and $z_3$ are parallel to each other. We consider the projection of the curve $C$ to a line that is perpendicular to these lines. By Lemma \ref{lemma:arc_tangent} (ii), this projection has length at least $\frac{k}{8\pi}$. Hence, $|z_3-z_1|\geq \frac{k}{8\pi}$. On the other hand,
\begin{align*}
|z_3-z_2| \leq t_3-t_2  \leq \frac{1}{k} \int_{t_2}^{t_3} \dot \theta =\frac{2}{k} (\pi/2- \alpha) <\frac{2}{k} \delta.
\end{align*}
Therefore, if we choose $\delta \leq \frac{ k^2}{32\pi}$, we have 
$$|z_1-z_2|\geq \frac{k}{8\pi}- \frac{2}{k} \frac{ k^2}{32 \pi}=\frac{k}{16\pi}\geq \frac{k}{16\pi } \length(C|[z_1,z_2]),$$
since $\length(C)=1$. 

Summarizing, if we choose $\delta =\min\{  \frac{k^2}{32\pi},\frac{\pi}{4}\}$, then  $C$ is an $L$-chord-arc curve with $L=\min\{ \cos(\pi/2-\delta), \frac{k}{16\pi} \}.$
\end{proof}

\bigskip

\subsection{Ahlfors $2$-regular regions}\label{section:ahlfors}
A Jordan region $D\subset \C$ is \textit{$M$-\textit{Ahlfors $2$-regular}} if for each $p\in \br{D}$ and for each $r\leq \diam(D)$ we have
\begin{align*}
\Area(B(p,r)\cap D) \geq M r^2.
\end{align*} 
This property is scale invariant, i.e., if $D$ is $M$-Ahlfors $2$-regular, then any scaled copy of $D$ is also $M$-Ahlfors $2$-regular. Examples of Ahlfors $2$-regular regions include disks and squares. On the other hand, regions with outward pointing cusps are not Ahlfors $2$-regular, and the lunes in Figure \ref{figure:lune} are not Ahlfors $2$-regular with uniform constants as $\diam(C)\to 0$. Ahlfors $2$-regular regions are very useful for counting arguments, such as in Lemma \ref{lemma:diameters_to_zero} below. 

\begin{lemma}\label{lemma:chordarc_area}
Let $C$ be an $L$-chord-arc curve and $D$ be the region enclosed by $C$. Then there exists a constant $M>0$ depending only on $L$ such that $D$ is $M$-Ahlfors $2$-regular. 
\end{lemma}
An $L$-chord-arc curve $C$ is a $L$-\textit{quasicircle}. That is, for any two points $a,b\in C$ we have $\diam(C|[a,b]) \leq L|a-b|$. The region $D$ enclosed by a quasicircle is called a \textit{quasidisk}, by definition. It was proved in  \cite[Corollary 2.3]{Schramm:transboundary} that quasidisks are Ahlfors  $2$-regular, {quantitatively}. This proves Lemma \ref{lemma:chordarc_area}.

\begin{lemma}\label{lemma:diameters_to_zero}
Let $\{D_i\}_{i\in \N}$ be a collection of disjoint, $M$-Ahlfors $2$-regular Jordan regions. Then, the following statements are true.
\begin{enumerate}[\upshape(i)]
\item Consider concentric balls $ B(x,r)$ and $ B(x,R)$, where $r<R$. For each $i\in \N$, if $\br{D_i}$ intersects $\partial B(x,r)$ and $\partial B(x,R)$, then 
$$\Area(D_i\cap B(x,R)\setminus B(x,r) )\geq c (R-r)^2,$$
where $c>0$ is a constant depending only on $M$. 
\item Let $E$ be a compact set. Then for each $\varepsilon>0$, the set $$\{i: D_i\cap E\neq \emptyset \,\, \textrm{and} \,\, \diam(D_i)>\varepsilon\}$$is finite.  
\item Let $E$ be a compact set. There exists a constant $K>0$ depending only on  $M$, such that for each $c>0$ the set 
$$\{i: D_i\cap E\neq \emptyset \,\, \textrm{and}\,\, \diam(D_i)\geq c \diam(E)\}$$
has at most $K(c^{-2}+1)$ elements.
\end{enumerate}
\end{lemma}
A proof of the first two parts of this lemma can be found in \cite{Ntalampekos:CarpetsThesis}. In particular, part (i) is proved in \cite[Remark 2.3.5]{Ntalampekos:CarpetsThesis} and part (ii) is proved in \cite[Lemma 2.3.4]{Ntalampekos:CarpetsThesis}.

\begin{proof}[Proof of $\mathrm{(iii)}$]
First, we bound the cardinality of the set 
$$I_4=\{i: D_i\cap E\neq \emptyset \,\, \textrm{and}\,\, \diam(D_i)\geq 4 \diam(E)\}.$$
Let $x\in E$, $r=\diam(E)$, and consider the ball $B(x,r)$, which contains $E$. If $D_i\cap E\neq \emptyset$ and $\diam(D_i)\geq 4\diam(E)$, then $\br{D_i}$ intersects $\partial B(x,r)$ and $\partial B(x,2r)$. By part (i) we have
$$\Area(D_i\cap (B(x,2r)\setminus B(x,r)))\geq c'r^2,$$
where $c'$ depends only on $M$. We have
\begin{align*}
4\pi r^2=\Area(B(x,2r)) \geq \sum_{i\in I_4}\Area(D_i\cap (B(x,2r)\setminus B(x,r))) \geq c'r^2 \cdot \# I_4,
\end{align*}  
where $\# I_4$ denotes the cardinality of $I_4$. This proves that $\# I_4 \leq 4\pi/c'$, which depends only on $M$. If $c\geq 4$, then the desired statement is proved, provided that we choose $K\geq 4\pi/c'$. 

Next, if $c<4$, it suffices to find a bound for the cardinality of the set 
$$I_c=\{i: D_i\cap E\neq \emptyset \,\, \textrm{and}\,\, c\diam(E)\leq \diam(D_i)<4 \diam(E)\}.$$
Let $x\in E$ and $R=5\diam (E)$. Then $B(x,R)$ contains $D_i$ whenever $i\in I_c$. Since each $D_i$ is $M$-Ahlfors $2$-regular, we have
\begin{align*}
\pi R^2=\Area(B(x,R)) \geq \sum_{i\in I_c} \Area(D_i) \geq M \diam(D_i)^2 \# I_c\geq Mc^2 \diam(E)^2\cdot \#I_c.
\end{align*}
This proves that $\#I_c\leq 25\pi M^{-1}c^{-2}$. Thus, if we choose $K=\max\{ 4\pi/c', 25\pi M^{-1}\}$, which depends only on $M$, then we have $\# I_c \leq Kc^{-2}$ and
\begin{align*}
\#\{i: D_i\cap E\neq \emptyset \,\, \textrm{and}\,\, \diam(D_i)\geq c\diam(E)\} \leq Kc^{-2}+K.
\end{align*}
This completes the proof.
\end{proof}

\begin{lemma}\label{lemma:segment_intersection}
Let $\{D_i\}_{i\in \N}$ be a collection of disjoint, $M$-Ahlfors $2$-regular Jordan regions and let $E$ denote the line segment $[0,l]\times \{0\}$, where $l>0$. Suppose that there exists a constant $c>0$ such that for each $i\in \N$ we have 
$$\diam(D_i)\leq cl \quad \textrm{and}\quad \dist(D_i,E) \leq c\diam(D_i).$$ 
Then there exists a constant $K>0$ that depends only on $c$ and $M$ such that for each $s\in(1,2]$ we have
\begin{align*}
\sum_{i=1}^\infty \diam(D_i)^s \leq \frac{K\cdot  l^s}{s-1}.
\end{align*}
\end{lemma}
See Figure \ref{figure:segment_intersection} for an illustration of the assumptions of the lemma.

\begin{figure}
	\begin{overpic}[width=.75\linewidth]{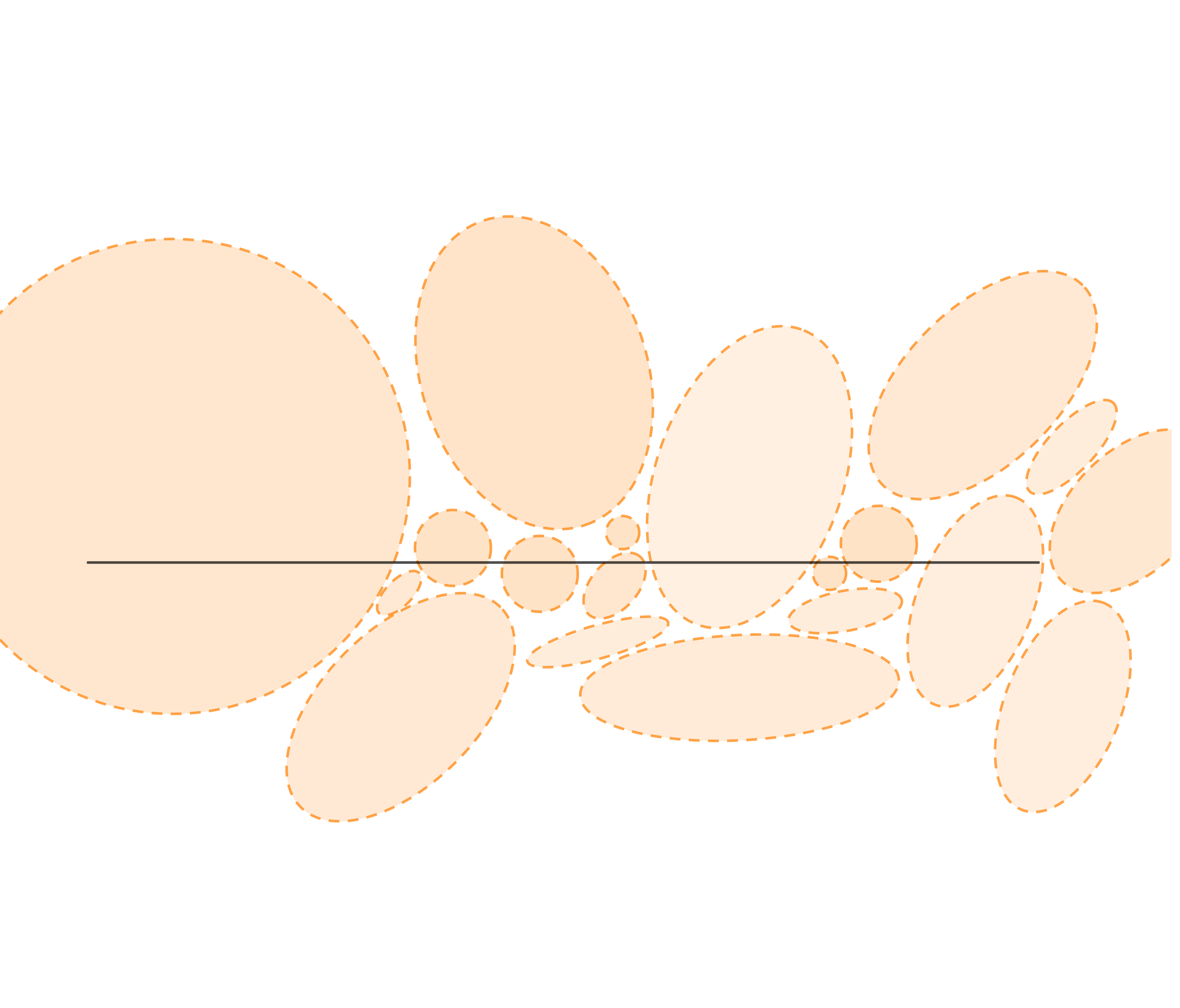}
		\put (10,34) {$E$}
	\end{overpic}
	\caption{The diameter of each $D_i$ is not much larger than the length of $E$ and the relative distance of $D_i$ to $E$ is bounded.}\label{figure:segment_intersection}
\end{figure}

\begin{proof}
By scaling the segment $E$ and the regions $D_i$, we may assume that $l=1$. Note that scaling does not affect the Ahlfors regularity constant $M$.

The assumption that $\dist(D_i,E) \leq c\diam(D_i)$ implies that the region $D_i$ is contained in the open $(1+c)\diam(D_i)$-neighborhood of $E$. For each $r>0$ we consider the family
$$I(r)=\{i\in \N:  cr< \diam(D_i)\leq  2cr\}.$$
We observe that $\bigcup_{r\in (0,1/2]} I(r)=\N$. Note that if $i\in I(r)$, then $D_i$ is contained in the open $(1+c)2cr$-neighborhood of $E$. Hence, there exists a constant $c_1>0$ depending only on $c$ such that for $r\in (0,1/2]$ we have
\begin{align*}
c_1r\geq  \area\left( \bigcup_{i\in I(r)} D_i \right)= \sum_{i\in I(r)} \area(D_i) \geq   \# I(r) \cdot M \cdot (cr)^2.
\end{align*}
It follows that for $r\in (0,1/2]$ we have $\#I(r) \leq c_2/r$ for a constant $c_2>0$ depending only on $c$ and $M$. Now for $s\in (1,2]$ we have
\begin{align*}
\sum_{i=1}^\infty \diam(D_i)^s =\sum_{k=1}^\infty \sum_{i\in I(2^{-k})} \diam(D_i)^s \leq c_2\sum_{k=1}^\infty 2^{-(k-1)s}\cdot c^s 2^k \leq \frac{c_3 }{1-2^{-(s-1)}}.
\end{align*}
where $c_3$ is a positive constant depending only on $c$ and $M$. Note that $1-2^{-(s-1)}\geq (s-1)/2$ for all $s\in [1,2]$ by concavity. The proof is complete.
\end{proof}

\bigskip

\section{Proof of main theorem}\label{section:proof_main}
Our proof follows the main steps of \cite{Larman:DimensionPackings}, but we have to adapt these steps to our generalized setting. In fact our first few considerations are very similar to the referenced result, but we have to deal with some extra complications related to the locality and generality of our result.   We suppose that $\mathcal P_\Omega= \{D_i\}_{i\in \N}$ is a packing such that for each $i\in \N$ the curve $\partial D_i$ has curvature bounded below by $k/\length(\partial D_i)$. By Theorem \ref{theorem:chordarc} all curves $\partial D_i$ are $L$-chord-arc curves, where $L$ depends only on $k$. Moreover, Lemma \ref{lemma:chordarc_area} implies that all regions $D_i$ are $M$-Ahlfors $2$-regular, where $M$ depends only on $k$. We split the proof in several subsections for the convenience of the reader. In Section \ref{section:initial_setup} we first make some reductions.

\bigskip

\subsection{Initial setup and definition of the function $g_m$}\label{section:initial_setup}
We first note that the residual set $\mathcal S$ is non-empty. To see this, we note that if $\partial D_1=\partial \Omega$, then $D_1=\Omega$, so the region $D_2$, which is contained in $\Omega$, intersects $D_1$, a contradiction. Hence, $ \partial D_1\cap \Omega\neq \emptyset$ and $\mathcal S\neq \emptyset$. 

Consider an open set $U\subset \Omega$ such that $U\cap \mathcal S\neq \emptyset$. Our goal is to show that
$$\dimh(U\cap \mathcal S)>K,$$
for a suitable constant $K>1$, depending only on $k$. 

\begin{claim}\label{claim:boundary_infinitely_many}
If $U$ is an open subset of $\Omega$, $U\cap \mathcal S\neq \emptyset$, and $U$ \textit{intersects} at most finitely many regions $D_i$, $i\in \N$, then $U\cap \mathcal S$ has non-empty interior.
\end{claim}

In this case $\dimh(U\cap \mathcal S)=2$ and there is nothing to prove. To prove the claim, note first that if $U$ does not intersect any region $D_i$ then $U\subset \mathcal S$ and the claim is true. Suppose now that at least one and at most finitely many regions $D_{i}$, $i\in I$, intersect $U$, but $U\cap \mathcal S$ has empty interior. By Lemma \ref{lemma:intersection_of_two} it follows that any two of the curves $\partial D_i$ can have at most one point in common.   Since $\partial (\bigcup_{i\in I} D_i )= \bigcup_{i\in I} \partial D_i$, it follows that $U\cap \mathcal S=U\setminus \bigcup_{i\in I}D_i=  \bigcup_{i\in I}(U\cap \partial D_i)$.
Let $j\in I$ and note that $U\cap \partial D_j\neq \emptyset$. Since $U$ is open, there exists an arc $J$ of $\partial D_j$ that is contained in $U$. By the above, at most finitely many points of $J$ can intersect some curve $\partial D_i$, $i\in I$, $i\neq j$. Hence, there exists a point $z\in J$ and a sufficiently small neighborhood $V\subset U$ of $z$ that does not intersect $D_i$, $i\in I$, $i\neq j$.  Therefore, $V\setminus {D_j} \subset U\setminus \bigcup_{i\in I} D_i$. Moreover, $V\setminus D_j$ has non-empty interior because $\partial D_j$ is a Jordan curve. We already have a contradiction and our claim is proved.

\begin{claim}\label{claim:interior_infinitely_many}
If $U$ is an open subset of $\Omega$, $U\cap \mathcal S\neq \emptyset$, and at most finitely many regions $D_i$ are \textit{contained} in $U$, then $U\cap  \mathcal S$ has non-empty interior. 
\end{claim}

Indeed, suppose that finitely many regions $D_i$ are contained in $U$ (but there could be infinitely many regions intersecting $U$). We restrict to a sufficiently small open set $U'\subset U$ with $U'\cap \mathcal S\neq \emptyset$ so that no region $D_i$ is contained in $U'$. Thus, if a region $D_i$ intersects $U'$, then it also intersects $\partial U'$ by connectivity. Let $V$ be an open set compactly contained in $U'$ such that $V\cap \mathcal S\neq \emptyset$. If $D_i$ intersects $V$, then it must also intersect $\partial U'$, so $\diam(D_i)>\dist(\partial U',V)>0$. By Lemma \ref{lemma:diameters_to_zero} (ii) there can be at most finitely many regions $D_i$ intersecting $V$. Claim \ref{claim:boundary_infinitely_many} now implies that $V\cap \mathcal S$ has non-empty interior.

From now on, we let $U$ be an open subset of $\Omega$ with $U\cap \mathcal S\neq \emptyset$. We consider an open square $V\subset \br V\subset  U$ such that $V\cap \mathcal S\neq \emptyset$. By Claim \ref{claim:interior_infinitely_many}, we may assume $V$ contains infinitely many regions $D_i$. We fix a region $D_j\subset V$. Then we consider a closed square $A\subset \br V$ containing $D_j$ such that 
$$\diam(D_j)\leq \diam(A)= \sqrt{2}\ell(A) \leq \sqrt{2}\diam(D_j),$$
where $\ell(A)$ denotes the side length of the square $A$; see Figure \ref{figure:tangential_square}.
\begin{figure}
	\begin{overpic}[width=.4\linewidth]{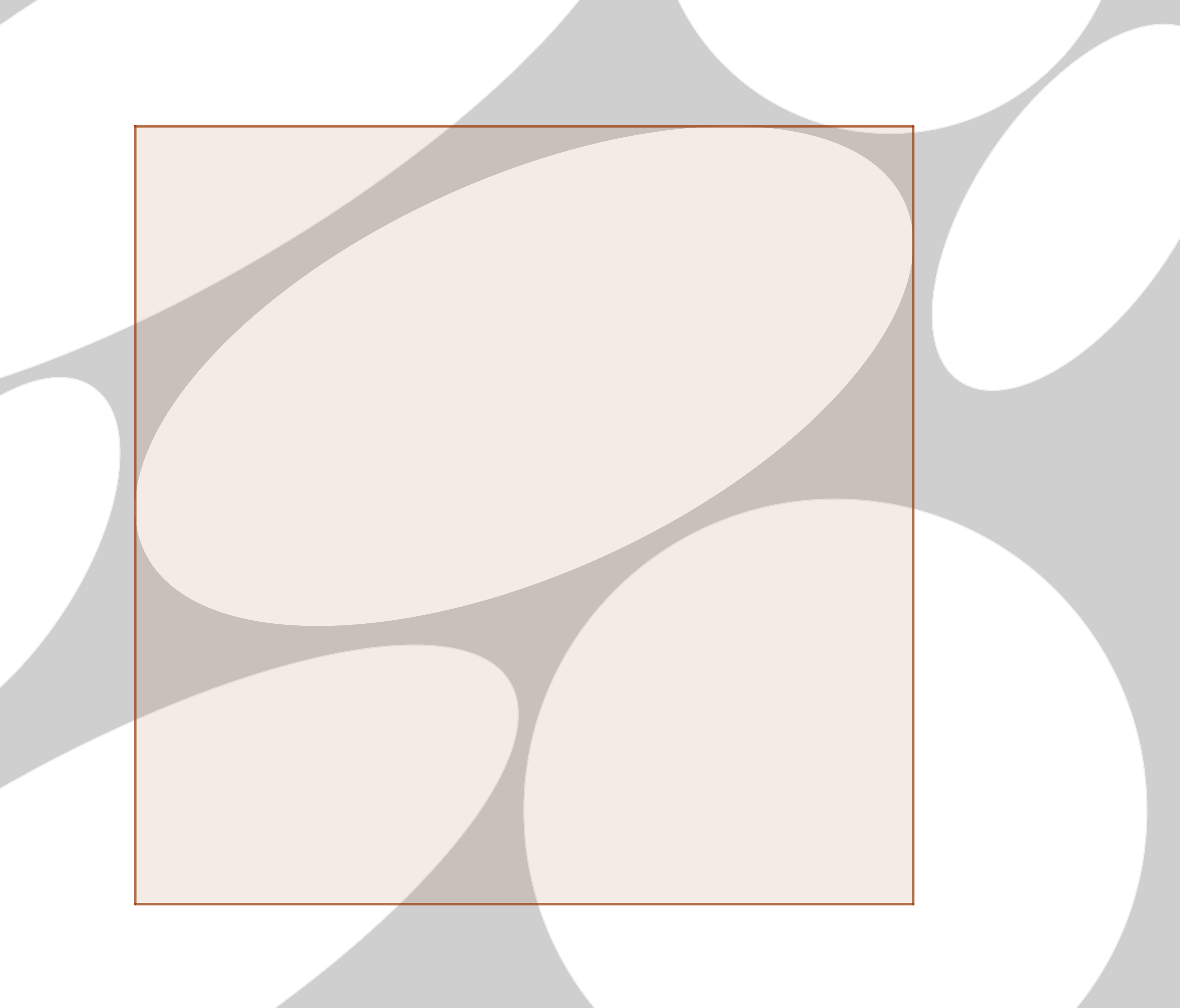}
		\put (42,52) {$B_{i_0}$}
	\end{overpic}
	\caption{The region $B_{i_0}=D_j$ contained inside the square $A$.}\label{figure:tangential_square}
\end{figure}
Note that $A\subset \br V\subset U$.  Therefore, it suffices to show that $$\dimh(A\cap \mathcal S)>K,$$
for a suitable constant $K>1$, depending only on $k$. 

After scaling and translating all involved sets, we may assume that $A$ is the unit square $[0,1]^2$. Moreover, by Claim \ref{claim:interior_infinitely_many} we may assume that there are infinitely many regions $D_i$ \textit{contained} in $\inter(A)$. We denote these by $B_i$, $i\in \N$. We remark that in Larman's setting in \cite{Larman:DimensionPackings} no regions $D_i$ are allowed to intersect the boundary of the unit square. We will see later in Section \ref{section:polygon_tilde} that in our case this possibility causes some complication that we have to overcome. By Lemma \ref{lemma:diameters_to_zero} (ii), we have
\begin{align*}
\lim_{i\to\infty}\diam(B_i)=0.
\end{align*}
Therefore, we may reorder the regions $B_i$ so that their diameters decrease, i.e., $\diam(B_i)\geq \diam(B_{i+1})$ for $i\in \N$. Note that by the choice of the square $A$ in the previous paragraph, one of the regions $B_i$, say $B_{i_0}$, is distinguished and has the property that
$$\sqrt{2}\diam(B_{i_0})\geq  \diam(A)=\sqrt{2}.$$
Since $\diam(B_1)\geq \diam(B_{i_0})$, it follows that $\diam(B_1)\geq 1$. The region $B_1$ is $M$-Ahlfors $2$-regular, so 
\begin{align}\label{theorem:b1_area}
\area(B_1)\geq M \diam(B_1)^2 \geq M.
\end{align}

We now begin to estimate the Hausdorff dimension of $S_\infty \coloneqq A\cap \mathcal S$. Consider a finite cover of the compact set $S_{\infty}$ by open squares $K_j$, $j\in \{1,\dots,p\}$. We may assume that none of the squares is contained in any region $D_i$. In order to prove the theorem, it suffices to show that there exist $s\in (1,2)$, depending only on $k$, and a constant $c>0$, depending only on $k$ and $s$, such that
\begin{align}\label{theorem:sum_diameter_bound}
\sum_{j=1}^p \diam(K_j)^s \geq c.
\end{align}
Recall the remarks after the definition of Hausdorff dimension in Section \ref{section:preliminaries}. In what follows we fix $s\in (1,2)$. Later we will impose more restrictions on $s$.

We define $S_n=S_\infty \cup \bigcup_{i=n+1}^\infty B_i$ for $n\in \N\cup \{0\}$. Note that there exists a $\delta>0$ such that $\bigcup_{j=1}^p K_j$ covers the open $\delta$-neighborhood of $S_\infty$. Hence, if $\diam({B_i})<\delta$, then the fact that $\partial B_i\subset S_\infty$ implies that ${B_i}\subset \bigcup_{j=1}^p K_j$. Since the diameters of $B_i$ tend to $0$ as $i\to\infty$, it follows that for all sufficiently large $n\in \N$ we have 
$$S_n\subset \bigcup_{j=1}^p K_j.$$
Throughout the proof $n\geq 1$ is fixed, so that the above inclusion holds. 

For $x\in [0,1]$ we denote by $L_x$ the vertical line passing through the point $(x,0)$. For each $x\in [0,1]$, the intersection of the line $L_x$ with $\bigcup_{j=1}^p K_j$, if non-empty, is a union of finitely many open segments $I(x,r)$ of length $\ell (x,r)$,  $r\in \{1,\dots,v(x)\}$.   We define 
$$g_n(x)= \sum_{r=1}^{v(x)} \ell(x,r)^{s-1}$$
for $x\in [0,1]$. Since $s-1<1$, we have
\begin{align*}
\ell(x,r)^{s-1}\leq \sum_{j:K_j\cap I(x,r)\neq \emptyset}  \length(L_x\cap K_j)^{s-1}.
\end{align*}
Therefore,
\begin{align*}
g_n(x)\leq \sum_{j=1}^{p} \length(L_x\cap K_j)^{s-1}.
\end{align*}
By integrating, we obtain
\begin{align}\label{theorem:length_bound_integral}
\int_0^1 g_n(x) \, dx \leq  \sum_{j=1}^p \ell(K_j)^{s},
\end{align}
where $\ell(K_j)$ denotes the side length of $K_j$. 

We also define a function $g_m(x)$ for $m\in \{0,\dots,n-1\}$ as follows. We append to the cover $\bigcup_{j=1}^p K_j$ of $S_n$ the regions $B_{m+1},\dots,B_n$ so that we obtain an open cover of $S_{m}$. The intersection of $L_x$ with that open cover is a union of finitely many  line segments $I(x,r,m)$ of length $\ell(x,r,m)$, $r\in \{1,\dots,v(x,m)\}$. We define
$$g_m(x)=\sum_{r=1}^{v(x,m)}\ell(x,r,m)^{s-1}.$$
Note that the function $g_m$ depends implicitly on $s$. A crude estimate for $g_m$ is obtained as follows. We observe that for $m<n$ we have $g_m(x)\geq \length(L_x\cap B_{m+1})^{s-1}$, so 
\begin{align*}
\int_0^1 g_m(x)\, dx &\geq \int_0^1 \length(L_x\cap B_{m+1})^{s-1}\, dx \geq \int_0^1 \length(L_x\cap B_{m+1})\, dx\\
&=\area(B_{m+1})
\end{align*}
since $s-1<1$ and the integrand is bounded above by $1$. In particular, $\int_0^1 g_0 \geq \area(B_1)\geq M$ by \eqref{theorem:b1_area}.

\bigskip

\subsection{Main estimate}\label{section:main_estimate}

If we could show that there exists $s>1$, depending only on $k$ such that $\int_0^1 {g_m} \geq \int_0^1 g_{m-1}$ for $m\in \{1,\dots,n\}$, then we would have that $\int_0^1 g_n \geq \int_0^1 g_0\geq M$, so $\sum_{j=1}^p \ell(K_j)^{s}\geq M$ by \eqref{theorem:length_bound_integral}. This would imply the desired \eqref{theorem:sum_diameter_bound}. In fact, we will show that $\int_0^1 {g_m}$ is \textit{essentially} larger than $\int_0^1 g_{m-1}$ in the following sense. We will prove that there exist some constant $c_0>0$, depending only on $s$ and $k$,  such that if $s$ is sufficiently close to $1$, depending only on $k$, then for all $m\in \{1,\dots,n\}$ we have
\begin{align}\label{theorem:main_estimate}
\int_0^1 g_m \geq \int_0^1g_{m-1} -c_0\diam( B_m)^s \cdot \# J_m,
\end{align}
where $J_m$ is the (possibly empty) set of indices $j\in \{1,\dots,p\}$ such that $B_m\cap \partial K_j\neq \emptyset$  and $\diam(K_j) \geq c_1 \length(\partial B_m)$ for a constant $c_1>0$, depending on $s,k$. Here $\# J_m$ denotes the cardinality of $J_m$. In other words, $J_m$ contains the indices $j$ such that $\partial K_j$ intersects $B_m$ and $K_j$ is relatively large. We remark that our main iterative estimate already departs from the corresponding estimate of Larman \cite[(35), p.~301]{Larman:DimensionPackings}.

We note that 
\begin{align*}
\sum_{m=1}^\infty \diam( B_m)^s\cdot  \# J_m=  \sum_{j=1}^p \sum_{m: j\in J_m} \diam( B_m)^s.
\end{align*}
Since $\dist(B_m, \partial K_j) =0$ and $\diam(B_m) \leq c_1^{-1}\diam(K_j) =c_1^{-1}\sqrt{2} \ell(K_j)$ for $j\in J_m$, we can apply Lemma \ref{lemma:segment_intersection} to each of the four edges of $\partial K_j$ and conclude that
$$\sum_{m: j\in J_m} \diam( B_m)^s \leq c_2 \ell(K_j)^s.$$
for some constant $c_2>0$ depending only on $s,k$. It follows that 
\begin{align*}
\sum_{m=1}^\infty  \diam( B_m)^s \cdot \#J_m \leq   c_2\sum_{j=1}^p \ell(K_j)^s.
\end{align*}

Assuming the estimates in \eqref{theorem:main_estimate}, by \eqref{theorem:length_bound_integral} we have
\begin{align*}
\sum_{j=1}^p \ell(K_j)^{s} \geq \int_0^1 g_n &\geq \int_0^1 g_{0} - c_0\sum_{m=1}^\infty  \diam( B_m)^s \cdot \#J_m \\
&\geq M - c_0c_2 \sum_{j=1}^p \ell(K_j)^s.
\end{align*}
Therefore,
\begin{align*}
\sum_{j=1}^p \ell(K_j)^{s} \geq M(1+c_0c_2)^{-1}
\end{align*}
and the latter is a positive constant depending only on $s$ and $k$. This completes the proof of \eqref{theorem:sum_diameter_bound} and thus of the main theorem. It remains to show the main estimate \eqref{theorem:main_estimate}. 

\bigskip

From now on, we fix $m\in \{1,\dots,n\}$ and our goal is to establish \eqref{theorem:main_estimate}.

\subsection{Basic estimate of the difference $g_m-g_{m-1}$}\label{section:basic_estimate}
We write
\begin{align*}
\int_0^1 g_m= \int_0^1g_{m-1} + \int_0^1(g_m-g_{m-1}).
\end{align*}
Our goal is to estimate the difference $g_m-g_{m-1}$ from below. Let $x\in [0,1]$ and consider the vertical line $L_x$. If $L_x$ does not intersect the region $B_m$, then we have $g_m(x)=g_{m-1}(x)$. Suppose, now, that $L_x$ intersects $B_m$. By convexity, $L_x\cap B_m$ is a line segment $J$. Observe that both of the endpoints of $J$ lie in $S_\infty\subset S_m$ and in particular they lie on $\partial B_m$.  The line segment $J$ intersects some of the segments $I(x,r,m)$, $r\in \{1,\dots,v(x,m)\}$. Recall that these line segments are the components of the intersection of $L_x$ with $\bigcup_{j=1}^p K_j\cup \bigcup_{i=m+1}^n B_i$. After renumbering, we assume that $J$ intersects the components $I(r)\coloneqq I(x,r,m)$, $r\in \{1,\dots,w\}$, and that these are ordered in an increasing fashion. This implies that $I(1)$ and $I(w)$ contain the endpoints of $J$; see Figure \ref{figure:basic_estimate}. Note that $I(1)$ and $I(w)$ could coincide, in which case we have $g_m(x)=g_{m-1}(x)$. We have the following estimate.

\begin{figure}
	\begin{overpic}[width=.75\linewidth]{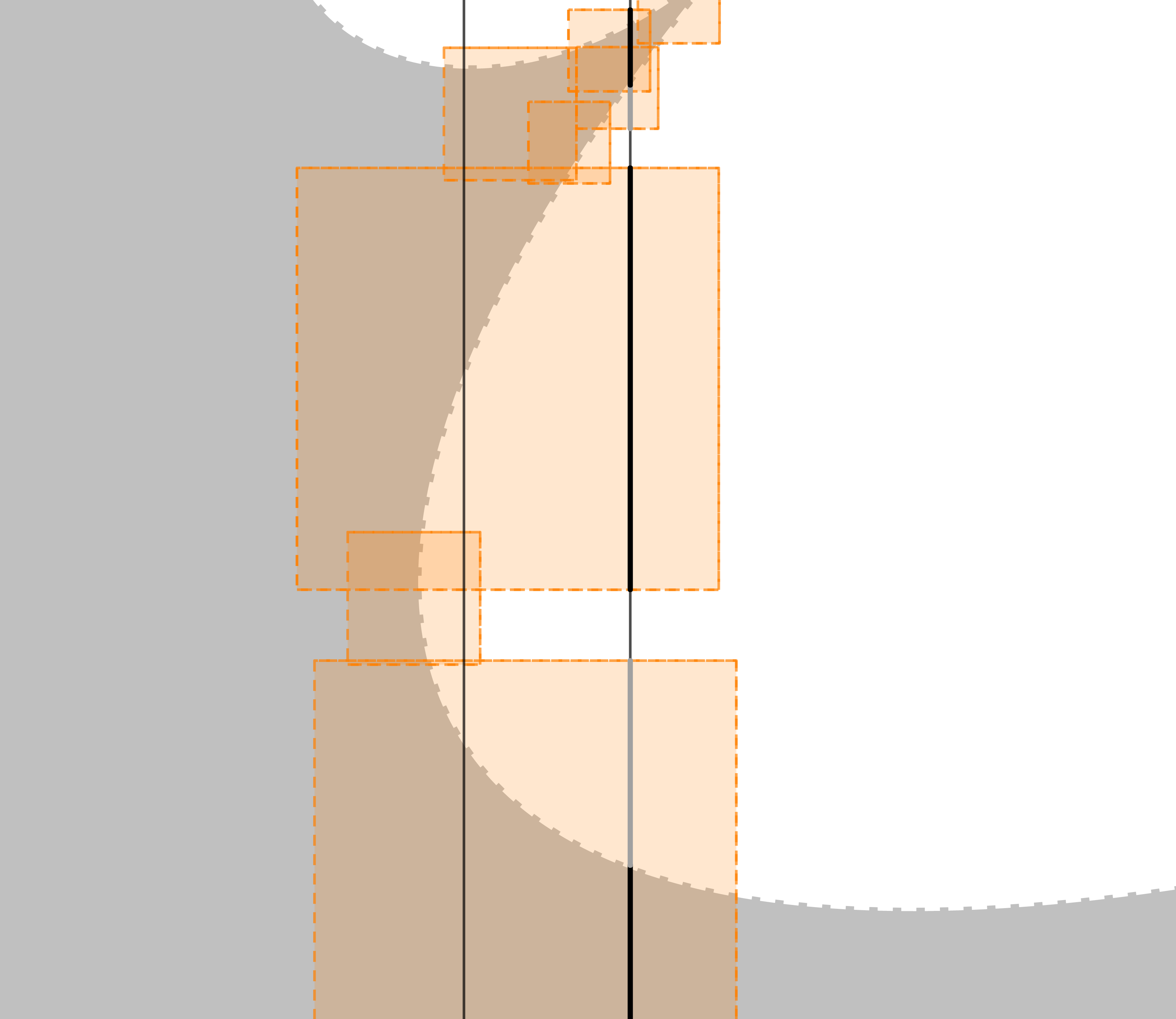}
		\put (80,50) {$B_m$}
		\put (55,8) {$\ell_{\text{out}}(1)$}
		\put (55,25) {$\ell_{\text{in}}(1)$}
		\put (55,60) {$\ell(2)$}
		\put (57,76) {$\ell_{\text{in}}(3)$}
		\put (57,80) {$\ell_{\text{out}}(3)$}
	\end{overpic}
	\caption{The intersection of two vertical lines with $\bigcup_{j=1}^p K_j\cup \bigcup_{i=m+1}^n B_i$. For the line on the left we have $w=1$ and $I(1)=I(w)$. For the line on the right we have $w=3$ and $I(1)\neq I(w)$.}\label{figure:basic_estimate}
\end{figure}
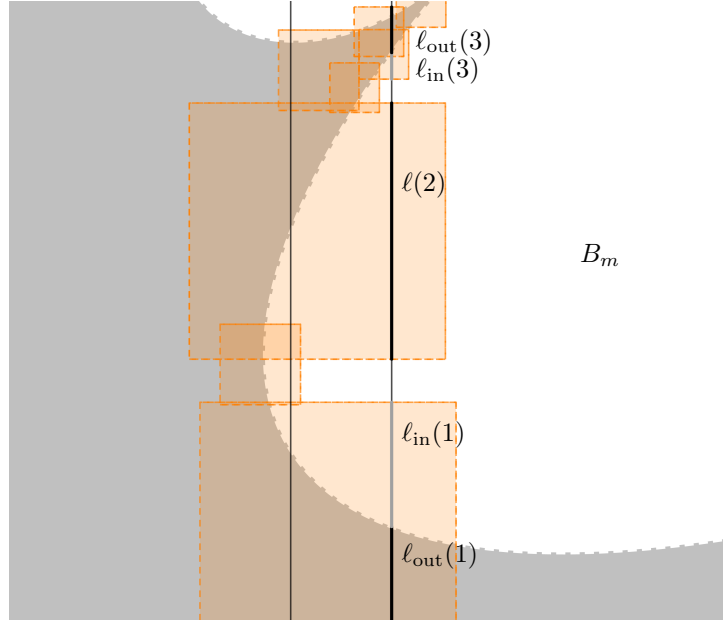

\begin{lemma}\label{lemma:basic_estimate}
Let $L=L_x$ be a vertical line with $L\cap B_m\neq \emptyset$ and consider the segments $I(1)$ and $I(w)$ as above, which are the components of $\bigcup_{j=1}^p K_j\cup \bigcup_{i=m+1}^n B_i$ that contain the endpoints of the segment $L\cap B_m$.
\begin{enumerate}[\upshape(i)]
\item If $I(1)=I(w)$, then $g_m(x)=g_{m-1}(x)$.
\item In general,
\begin{align*}
g_m(x)-g_{m-1}(x) \geq - \length(L_x\cap B_m)^{s-1}.
\end{align*}
\item If $I(1)\neq I(w)$, then
\begin{align*}
g_m(x)-g_{m-1}(x) &\geq \length(I_1)^{s-1} + \length(I_2)^{s-1}\\
&\qquad\qquad-\length(I_1\cup I_2\cup (L_x\cap B_m))^{s-1},
\end{align*}
where $I_1$ and $I_2$ are any subsegments of $I(1)\setminus B_m$ and $I(w)\setminus B_m$, respectively.
\end{enumerate}
\end{lemma}

\begin{proof}
As already pointed out, if $I(1)=I(w)$, then $g_m(x)=g_{m-1}(x)$, i.e., (i) holds. Moreover, the estimate in (ii) holds trivially in this case. Henceforth, we assume that $I(1)\neq I(w)$. We denote by $\ell(r)$ the length of $I(r)$. We have $\ell(1)=\ell_{\mathrm{in}}(1)+ \ell_{\mathrm{out}}(1)$, where $\ell_{\mathrm{in}}(1)$ is the length of $I(1)\cap B_m$ and $\ell_{\mathrm{out}}(1)$ is the length of $I(1)\setminus B_m$; see Figure \ref{figure:basic_estimate}. Similarly, we have $\ell(w)=\ell_{\mathrm{in}}(w)+ \ell_{\mathrm{out}}(w)$. We have
\begin{align}\label{theorem:difference_gm}
\begin{aligned}
g_m(x)-g_{m-1}(x)&= ( \ell_{\mathrm{in}}(1)+ \ell_{\mathrm{out}}(1))^{s-1} + \ell(2)^{s-1}+\dots+\ell(w-1)^{s-1} \\ &\qquad\qquad + (\ell_{\mathrm{in}}(w)+  \ell_{\mathrm{out}}(w))^{s-1}\\
&\qquad\qquad- (\ell_{\mathrm{out}}(1) +\length(L_x\cap B_m) +\ell_{\mathrm{out}}(w))^{s-1}.
\end{aligned}
\end{align}
Consider the auxiliary function $$h(a,b)=(c_1+a)^{s-1}+c_2+(c_3+b)^{s-1}-(a+\ell +b)^{s-1},$$ 
where $c_1=\ell_{\mathrm{in}}(1)$, $c_2=\ell(2)^{s-1}+\dots+\ell(w-1)^{s-1}$, $c_3=\ell_{\mathrm{in}}(w)$, and $\ell=\length(L_x\cap B_m)$. We note that for $a,b\geq 0$ we have
\begin{align*}
\frac{\partial h}{\partial a}\geq 0 \qquad \textrm{and} \qquad \frac{\partial h}{\partial b}\geq 0,
\end{align*}
since $c_1,c_3\leq \ell$. Therefore, the right-hand side of \eqref{theorem:difference_gm} becomes smaller if we replace $\ell_{\mathrm{out}}(1)$ and $\ell_{\mathrm{out}}(w)$ with smaller quantities. Note that if $I_1$ and $I_2$ are subsegments of $I(1)\setminus B_m$ and $I(w)\setminus B_m$, respectively, then $\length(I_1)\leq \ell_{\mathrm{out}}(1)$ and $\length(I_2)\leq \ell_{\mathrm{out}}(w)$. Therefore,
\begin{align*}
g_m(x)-g_{m-1}(x)&\geq ( \ell_{\mathrm{in}}(1)+ \length(I_1))^{s-1} + \ell(2)^{s-1}+\dots+\ell(w-1)^{s-1} \\ &\qquad\qquad + (\ell_{\mathrm{in}}(w)+  \length(I_2))^{s-1}\\
&\qquad\qquad- (\length(I_1) +\length(L_x\cap B_x) +\length(I_2))^{s-1}\\
&\geq  \length(I_1)^{s-1} + \length(I_2)^{s-1}\\
&\qquad\qquad-\length(I_1\cup I_2\cup (L_x\cap B_m))^{s-1}.
\end{align*}
This completes the proof of (iii). Note that part (ii) also follows immediately by taking $\length(I_1)=\length(I_2)=0$.
\end{proof}

\bigskip

\subsection{Restriction to a polygon around $B_m$}\label{section:polygon}
In this subsection we will construct an open polygonal region $P$ that circumscribes $B\coloneqq B_m$ such that $P\setminus B$ does not intersect any regions $D_l$ with diameter larger than or equal to the diameter of $B$. Note that the polygonal region $P$ might intersect only regions $B_l$, which are contained in the unit square, or it might also intersect some of the regions $D_l$ that intersect the boundary of the unit square and thus are not contained in the collection $B_l$, $l\in \N$. Then, in Section \ref{section:polygon_tilde}, we will refine the polygon $P$  and construct another polygon $\widetilde P \subset P$ that also circumscribes $B$ such that the intersection of $\widetilde P$ with regions $D_l$ that intersect $\partial ([0,1]^2)$ is negligible in a sense. Therefore, $\widetilde P$ essentially intersects only regions $B_l$, $l\in \N$. This will be crucial for the application of the estimates of Lemma \ref{lemma:polygon_good_estimate} below.
 
For a set $E\subset \C$ we denote the projection of $E$ to the $x$-axis by $\proj(E)$. Moreover, for $\mu>0$ we define the \textit{$\mu$-strip of $E$} to be the open set of points in the plane whose projection to the $x$-axis has distance less than $\mu$ from $\proj(E)$.

Consider the open rectangle $R$ that circumscribes $B$ and has sides parallel to the coordinate axes. Note that each side of the rectangle intersects $\partial B$ in precisely one point, by the strict convexity of $\partial B$. Let $P$ be an open polygonal region that \textit{circumscribes} $B$ with $B\subset P\subset R$; see Figure \ref{figure:construction_P}. By definition, each edge of $\partial P$ is tangent to $\partial B$. The convexity of $B$ implies that $P$ is also convex. We denote by $Z(P)$ the finite set of the points of tangency between $\partial P$ and $\partial B$. Note that $Z(P)\supset Z(R)$ and in general if $P'\subset P$ is another polygonal region that circumscribes $B$, then $Z(P')\supset Z(P)$. 

\begin{figure}
\centering
\begin{minipage}{.49\linewidth}
	\centering
	\begin{overpic}[width=0.9\linewidth]{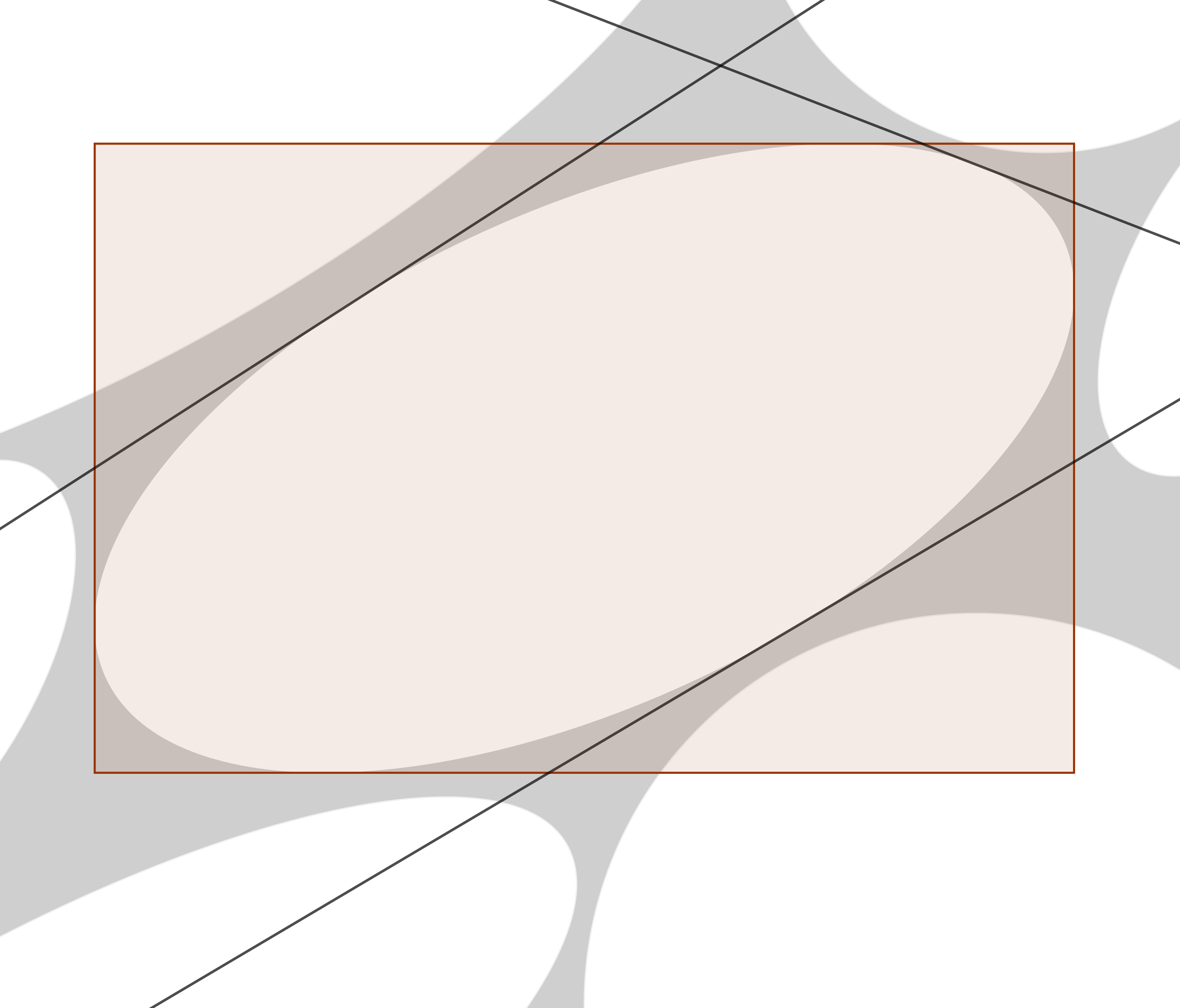}
		\put (47,43) {$B$}
	\end{overpic}
	
\end{minipage}
\begin{minipage}{.49\linewidth}
	\centering
	\begin{overpic}[width=0.9\linewidth]{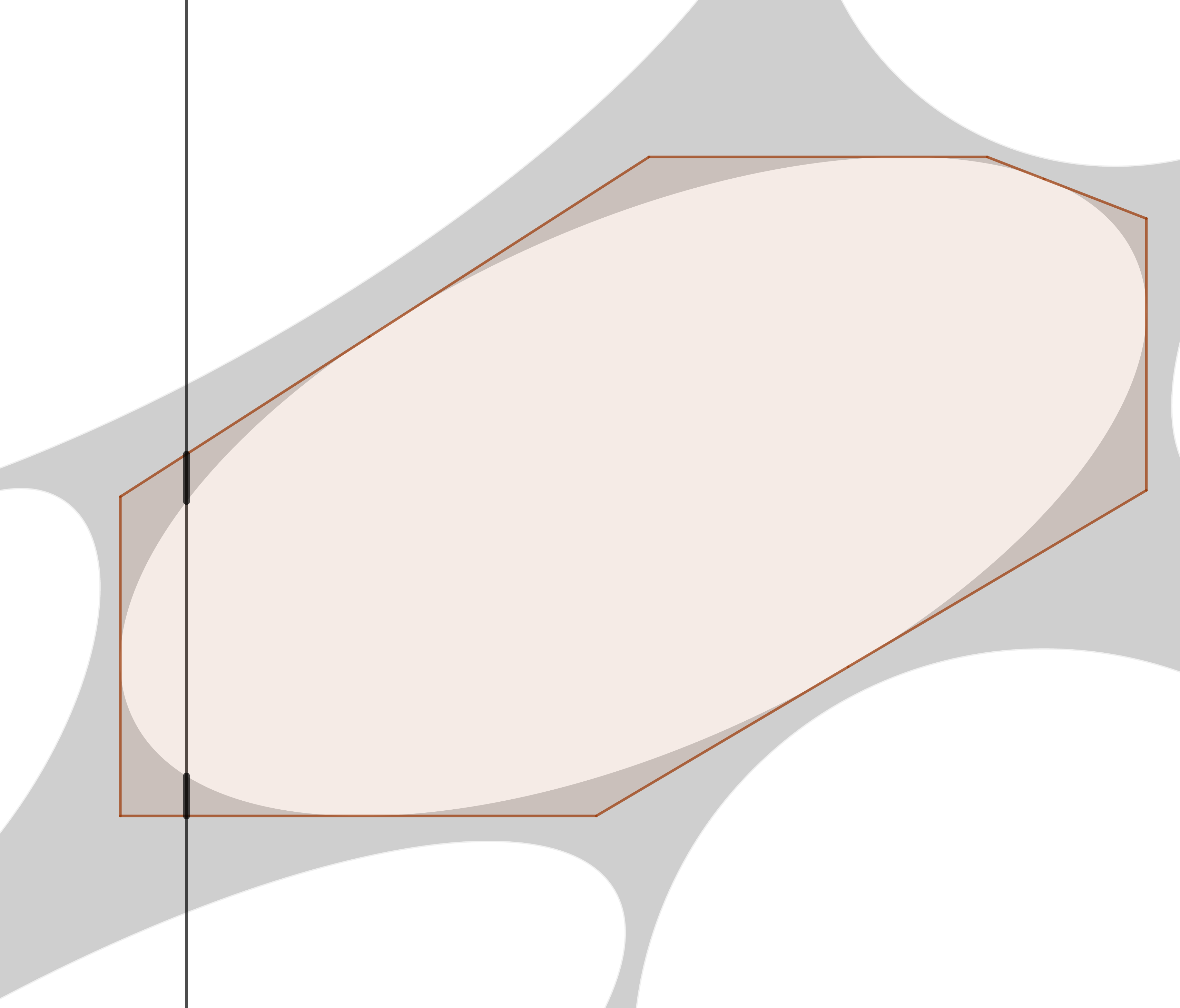}
		\put (47,43) {$B$}
		\put (17, 43) {$I_2$}
		\put (17, 17) {$I_1$}
		\put (8, 40) {$z_2$}
		\put (14.6, 42) {{\LARGE\textbf{.}}}
		\put (7, 48) {$w_1$}
		\put (14.6, 47) {{\LARGE\textbf{.}}}
		\put (37, 56) {$z_1$}
		\put (35, 59.3) {{\LARGE\textbf{.}}}
	\end{overpic}
\end{minipage}

\caption{Construction of a polygon $P\subset R$ that circumscribes $B$. First, $B$ is separated from some regions $D_i$ by lines tangent to $\partial B$. Then $P$ is formed by these lines and the sides of $R$.}\label{figure:construction_P}

\end{figure}
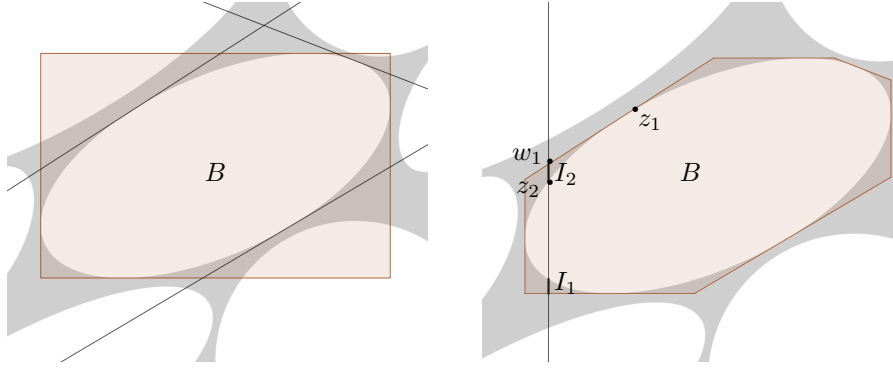

\begin{lemma}\label{lemma:polygon_good_estimate}
Consider a polygonal region $P\subset R$ that circumscribes $B$. Let $L=L_x$ be a vertical line with $L\cap B\neq \emptyset$ and denote by $I_i$, $i=1,2$, the components of $L\cap(P\setminus B)$. Finally, let $\lambda>0$ and suppose that $L$ does not intersect the $\lambda \cdot \length(\partial B)$-strip of the set $Z(P)$. Then the following statements are true.
\begin{enumerate}[\upshape(i)]
\item For $i=1,2$ we have
\begin{align*}
\frac{\length(I_i)}{\length(\partial B)} \geq \frac{k}{2\pi}\lambda^2.
\end{align*} 
\item For each $p>0$ there exists a constant $s_0>1$ depending only on $k$ and $p$ such that if $s\in (1,s_0)$ and $\lambda=(s-1)^p$, then
\begin{align*}
\length(I_1)^{s-1}+\length(I_2)^{s-1}- \length(L\cap P)^{s-1} \geq \frac{1}{2} \length(\partial B)^{s-1}.
\end{align*}
\item Suppose that $s$ and $\lambda$ are as in \textup{(ii)} and that
\begin{enumerate}
\item[\upshape(iii-1)] $I_i$, $i=1,2$, does not intersect any region $D_l$ with $\diam(D_l)\geq \diam(B)$, and
\item[\upshape(iii-2)] $I_i$, $i=1,2$, does not intersect any region $D_l$ with  $D_l\cap \partial ([0,1]^2)\neq \emptyset$.
\end{enumerate}
If $I_1$ and $I_2$ lie in different components of $L_x\cap (\bigcup_{j=1}^p K_j \cup \bigcup_{i={m+1}}^n B_i)$, then
\begin{align*}
g_m(x)-g_{m-1}(x) \geq \frac{1}{2}\length(\partial B)^{s-1} >0,
\end{align*}
otherwise,
$$g_m(x)=g_{m-1}(x).$$
\end{enumerate}
\end{lemma} 

See Figure \ref{figure:construction_P} for an illustration of the intersection $L\cap (P\setminus B)$ and Figure \ref{figure:basic_estimate} for the intersection $L_x\cap (\bigcup_{j=1}^p K_j \cup \bigcup_{i={m+1}}^n B_i)$ that appears in part (iii).

\begin{proof}
By assumption, the curvature of $\partial B$ is bounded below by $k/\length(\partial B)$. Note that the conclusion of part (i) is scale invariant. Hence, by scaling, we may assume that $\length(\partial B)=1$ and that the curvature of $\partial B$ is bounded below by $k>0$. Since the vertical line segment $\br{I_i}$ does not intersect the set $Z(P)$, it follows that $\br{I_i}$ intersects an edge $E$ of the polygon $P$ at a point $w_1$ that is not a point of tangency of $E$ with $\partial B$. The edge $E$ is tangent to $\partial B$ at a point $z_1=\gamma(t_1)\in Z(P)$. Hence, $E$ is contained in the tangent line $L(z_1)$ of $\partial B$ at $z_1$ and $w_1\in L(z_1)$.  We let $z_2\in \partial B$ be the point of intersection of $\br{I_i}$ with $\partial B$; see Figure \ref{figure:construction_P}. By Lemma \ref{lemma:arc_tangent} (i) we have
\begin{align*}
\diam(I_i)=|z_2-w_1|\geq \dist(z_2,L(z_1)) \geq \frac{k}{2\pi} \length(\partial B|[z_1,z_2])^2 \geq \frac{k}{2\pi}|z_1-z_2|^2. 
\end{align*}
By assumption, $\re(z_2-z_1) \geq \lambda$, so the desired conclusion follows.

Next we prove the second part of the lemma. Again, since the statement is scale invariant, we may assume that $\length(\partial B)=1$. Note that $\length(L\cap P)\leq \diam(B)\leq \length(\partial B)=1$, since $P$ is contained in the circumscribing rectangle $R$. By the first part of the lemma we have
\begin{align*}
\length(I_1)^{s-1}+\length(I_2)^{s-1}- \length(L\cap P)^{s-1} \geq  2(k/2\pi)^{s-1}\lambda^{2s-2}-1.
\end{align*}
If we set $\lambda=(s-1)^p$, then  $2(k/2\pi)^{s-1}\lambda^{2s-2}-1\to 1$ as $s\to 1^+$. The conclusion follows. 

For (iii), we note first that the segments $I_i$, $i=1,2$, are contained in $\bigcup_{j=1}^p K_j\cup\bigcup_{i=m+1}^n B_i$. Indeed, by assumption (iii-2) we have $I_i\subset \bigcup_{j=1}^p K_j\cup\bigcup_{l=1}^n B_l$ since $I_i$ does not intersect any region $D_l$ that intersects the boundary of the unit square. Moreover, by (iii-1), $I_i$ does not intersect any region $D_l$ with diameter larger than or equal to the diameter of $B$. This implies that $I_i\cap B_l=\emptyset$ for $l\leq m$; recall the enumeration of the regions $B_l$, $l\in \N$, by decreasing diameters. 

If $I_1$ and $I_2$ lie in the same component of $L_x\cap (\bigcup_{j=1}^p K_j \cup \bigcup_{i={m+1}}^n B_i)$, then we have $I(1)=I(w)$, using the notation of Lemma \ref{lemma:basic_estimate}, and $g_m(x)=g_{m-1}(x)$ by part (i) of the aforementioned lemma. Otherwise, by Lemma \ref{lemma:basic_estimate} (iii), and using part (ii) of the current lemma, we conclude that 
$$g_m(x)-g_{m-1}(x)\geq \frac{1}{2} \length(\partial B)^{s-1}.$$
The proof is complete.
\end{proof}

Now, we construct a polygonal region $P$  that circumscribes $B$ such that $P\setminus B$ does not intersect any regions $D_i$ with diameter larger than or equal to the diameter of $B$; in this case the assumption (iii-1) of the previous lemma is always satisfied. By Lemma \ref{lemma:diameters_to_zero} (iii) there exists a positive integer $N_0$, depending only on $k$, such that if $D_i$, $i\in I_R$, is the family of regions intersecting the rectangle $R$ and having diameter larger than or equal to $\diam(B)$, then $\# I_R\leq N_0$. Since all regions $D_i$ are convex, for each $i\in I_R$ there exist a line $L_i$ tangent to $\partial B$ that separates $D_i$ from $B$; this follows from Minkowski's hyperplane separation theorem. The lines $L_i$, $i\in I_R$, together with the sides of the rectangle $R$ define a polygon $P$ with at most $N_0+4$ sides that circumscribes $B$; see Figure \ref{figure:construction_P}. Note that the cardinality of the set $Z(P)$ is bounded by $N_0+4\eqqcolon N$.

\bigskip

\subsubsection{Refinement of the polygon $P$}\label{section:polygon_tilde}
Note that the estimate from part (iii) of Lemma \ref{lemma:polygon_good_estimate}, if it were true for all $x\in [0,1]$, would imply the desired main estimate \eqref{theorem:main_estimate}---in fact, a stronger version of it without the subtracted term. 

In order to apply the favorable estimate from part (iii) of  Lemma \ref{lemma:polygon_good_estimate}, we will construct a smaller polygon $\widetilde P\subset P\subset R$ that circumscribes $B$ so that $I_1$ and $I_2$ (i.e., the components of $L\cap (\widetilde P\setminus B)$) do not intersect any regions $D_l$ with $D_l\cap \partial ([0,1]^2)\neq \emptyset$, whenever the vertical line $L$ does not intersect the $\lambda\length(\partial B)$-strip of $Z(R)$; recall that $Z(R)$ is the set points of tangency between $\partial R$ and $\partial B$. We remark that the assumptions of  Lemma \ref{lemma:polygon_good_estimate} require that the vertical line $L$ avoid a strip of $Z(\widetilde P)$, while here we are working with a strip of $Z(R)$, which is a subset of $Z(\widetilde P)$. In the next subsection, we will restrict to vertical lines avoiding a strip of $Z(\widetilde P)$ as well. 

We fix $\lambda\in (0,1)$. We consider the rectangle $\widetilde R \subset R$ whose top and bottom sides are at distance $(k/2\pi)\lambda^2 \length(\partial B)$ from the top and bottom sides of $R$ and whose left and right sides are at distance $\lambda\length(\partial B)$ from the left and right sides of $R$, respectively; see Figure \ref{figure:polygon_PR}. Note that by Lemma \ref{lemma:arc_tangent} (ii) the length of the horizontal and vertical sides of $R$ is at least $c\length(\partial B)$ for some $c>0$ depending only on $k$. Hence, if $\lambda$ is sufficiently small, depending only on $k$, then the construction of $\widetilde R$ is possible. 

By Lemma \ref{lemma:polygon_good_estimate} (i), applied to the polygon $R$, we know that $R\setminus \widetilde R$ does not intersect $B$, except at the $\lambda\length(\partial B)$-strip  of $Z(R)$, which we denote by $S$. The strip $S$ has at most $4$ components. Therefore, 
$$(R\setminus \widetilde R )\setminus S $$
consists of a uniformly bounded number of rectangles (in fact, at most $6$), none of which intersects $B$; see Figure \ref{figure:polygon_PR}. We separate each of these rectangles from $B$ with a line tangent to $\partial B$, thus creating a new convex polygon $P'\subset P$ that circumscribes $B$ and has a uniformly bounded, say by $N'\in \N$, number of edges. We note that $P'\subset \widetilde R \cup S$. See Figure \ref{figure:polygon_Ptilde}.

\begin{figure}
\centering
\begin{minipage}{0.49\linewidth}
	\centering
	\begin{overpic}[width=0.9\linewidth]{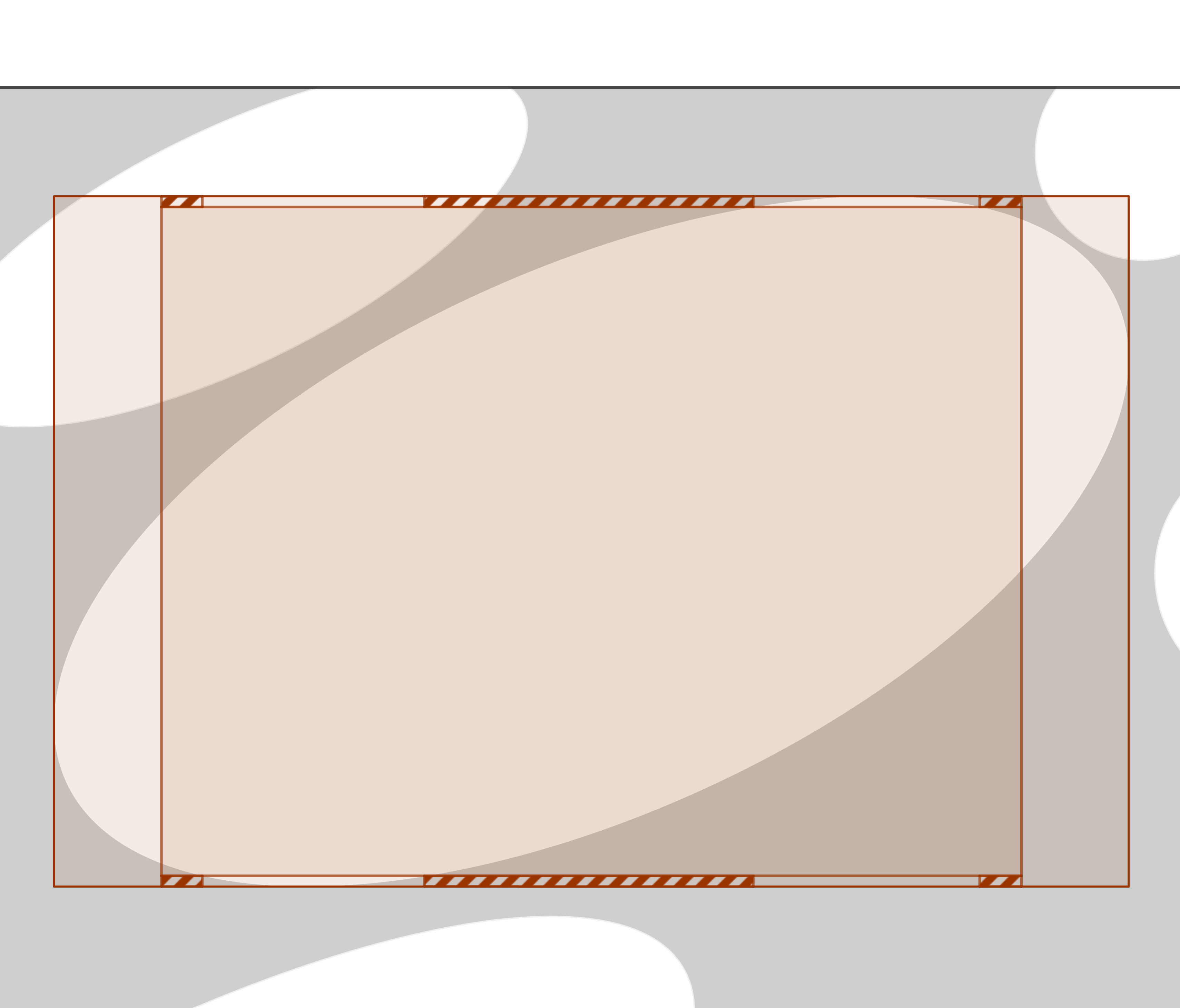}
		\put (2,83) {$\partial([0,1]^2)$}
		\put (48,34) {$B$}
		\put (14, 16) {$\widetilde{R}$}
		\put (5, 12) {$R$}
	\end{overpic}
\end{minipage}
\begin{minipage}{0.49\linewidth}
	\centering
	\begin{overpic}[width=0.9\linewidth]{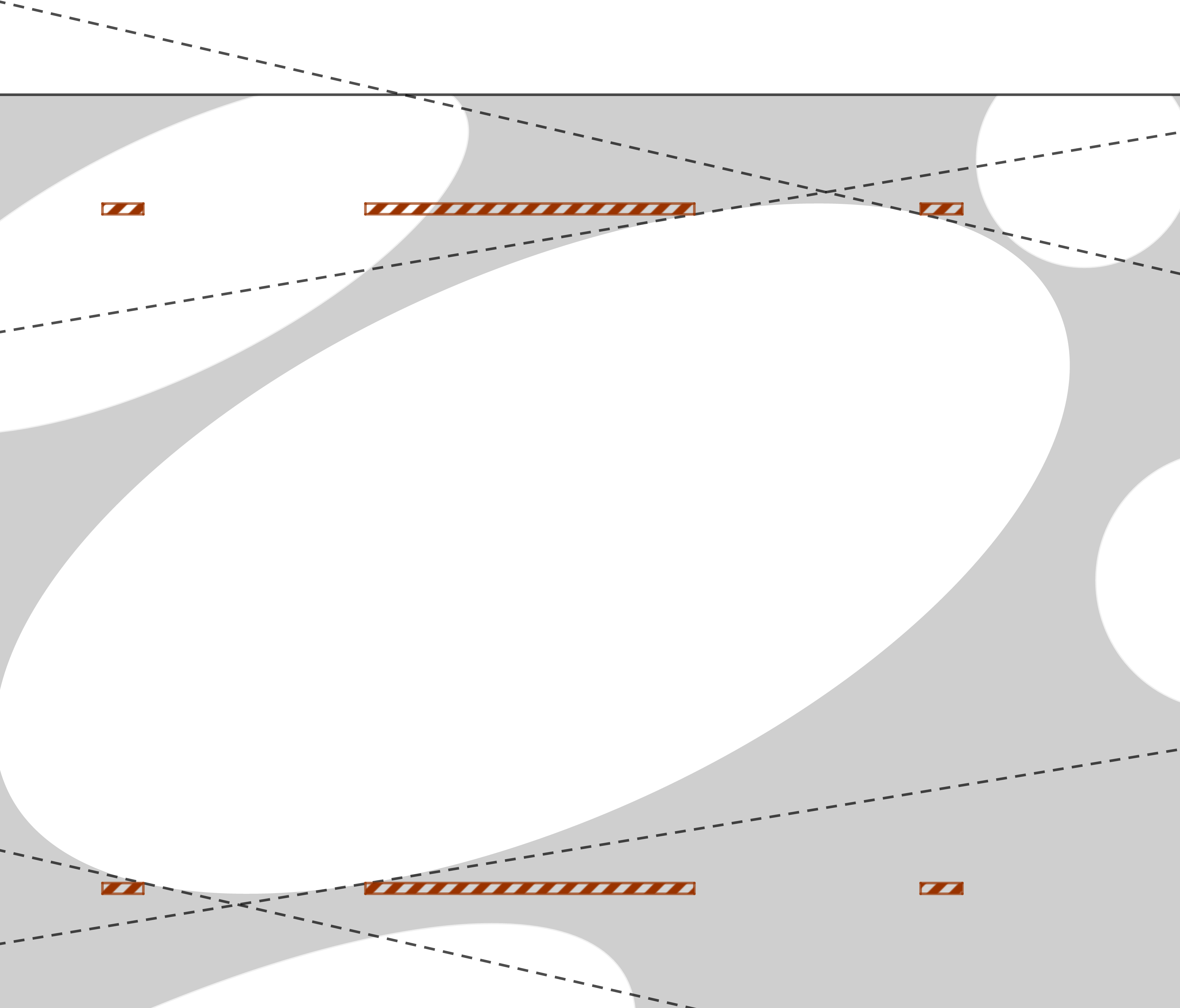}
		\put (2,81) {$\partial([0,1]^2)$}
		\put (42,36) {$B$}
	\end{overpic}
\end{minipage}
\caption{The polygon $P$ is assumed in this figure, for simplicity, to be the circumscribing rectangle $R$. Thus, all regions $D_l\neq B$ that intersect $R$ already have diameter smaller than $\diam(B)$. The strip $S$ has four components in this case. The rectangle $\widetilde R\subset R$ is illustrated and the set $(R\setminus \widetilde R)\setminus S$ consists of the six striped rectangles. In the right figure, these rectangles are separated from $B$ by tangent lines.}\label{figure:polygon_PR}
\vspace{1em}
	
\begin{minipage}{0.32\linewidth}
\centering
	\begin{overpic}[width=0.9\linewidth]{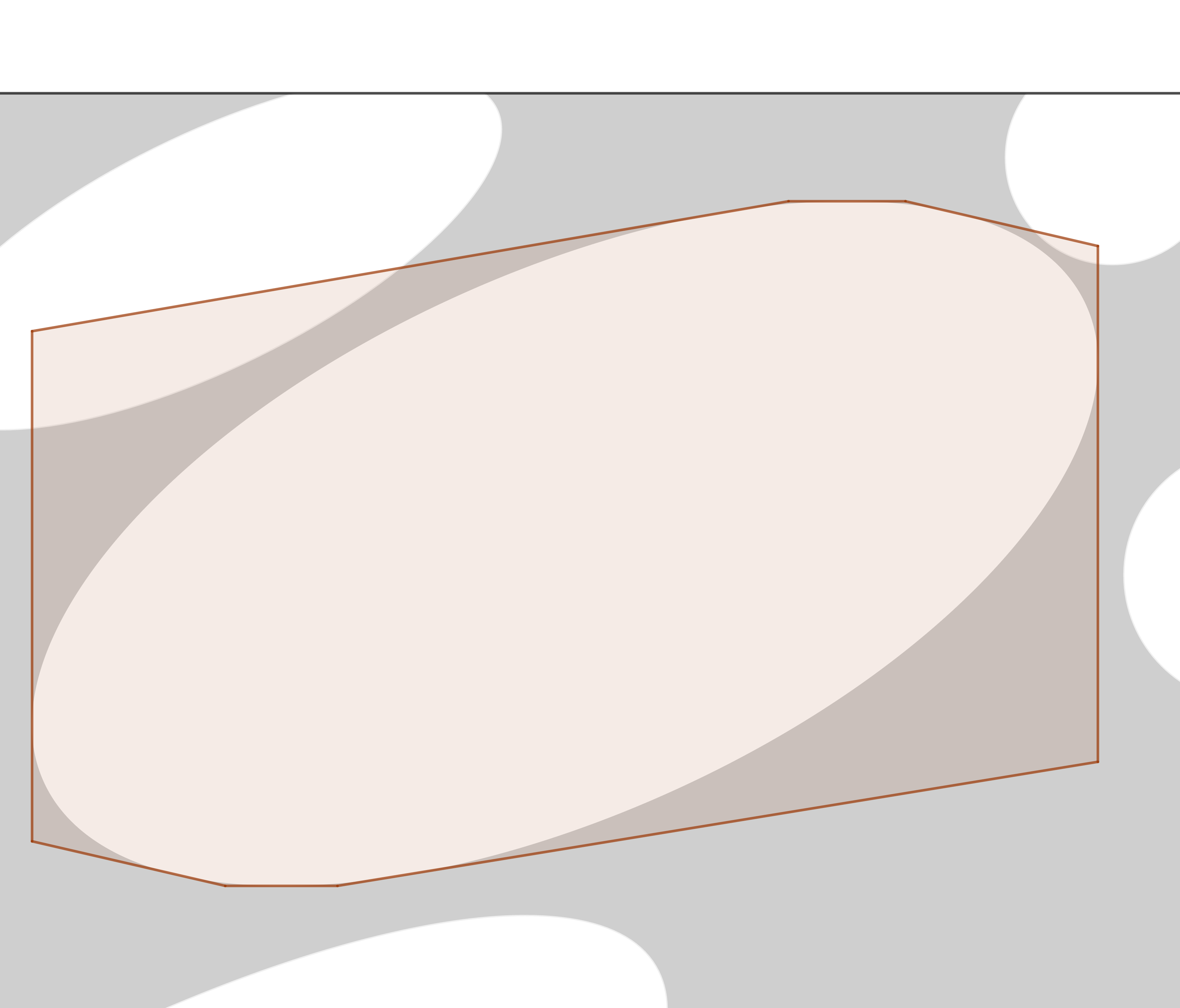}
		\put (2,83) {$\partial([0,1]^2)$}
		\put (46,36) {$P^\prime$}
	\end{overpic}
\end{minipage}
\begin{minipage}{0.33\linewidth}
\centering
	\begin{overpic}[width=.9\linewidth]{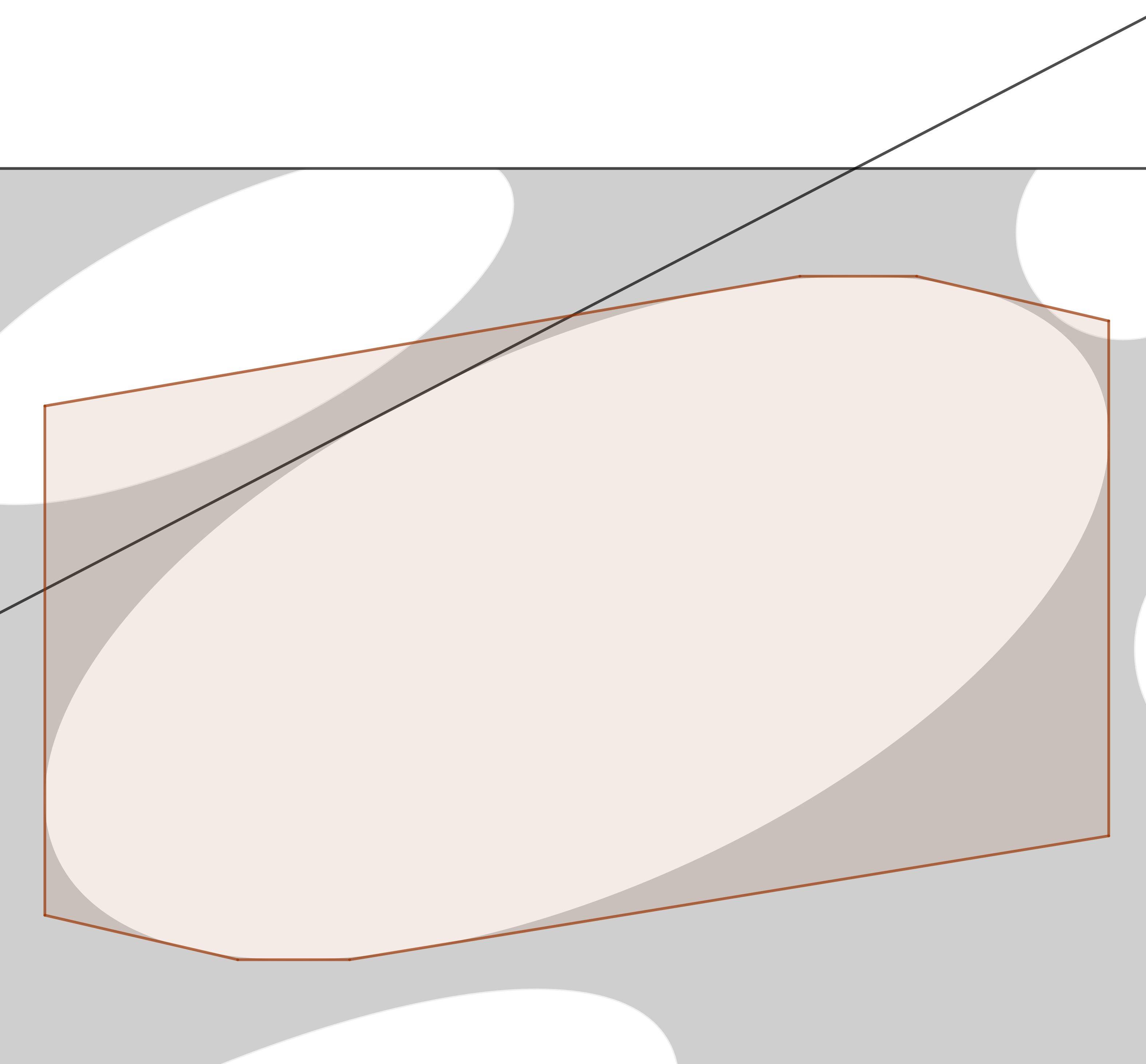}
		\put (2,83) {$\partial([0,1]^2)$}
		\put (48,36) {$P^\prime$}
		\put (16,64) {$D_l$}
	\end{overpic}
\end{minipage}
\begin{minipage}{0.33\linewidth}
\centering
	\begin{overpic}[width=.9\linewidth]{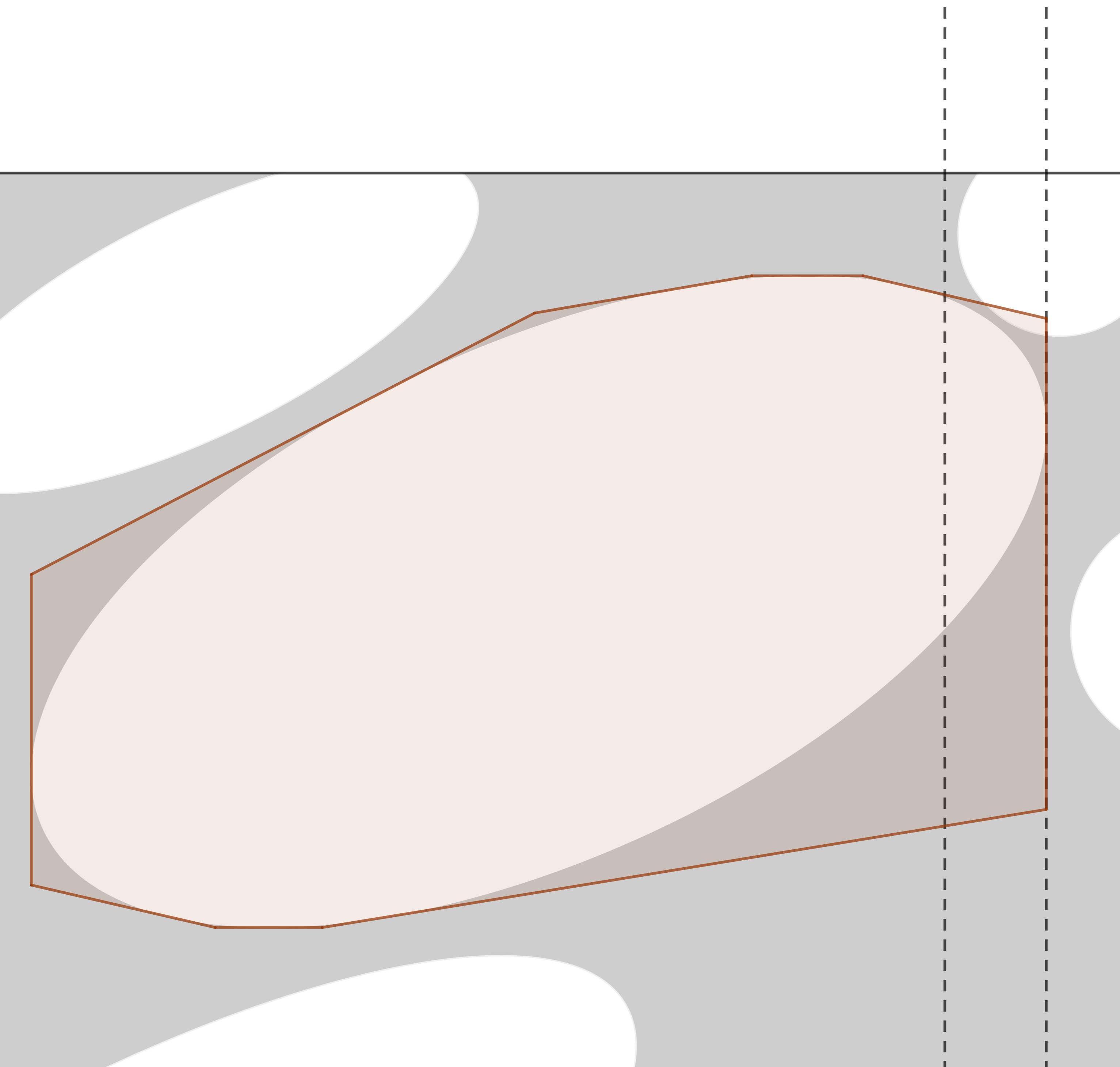}
		\put (2,84) {$\partial([0,1]^2)$}
		\put (46,38) {$\widetilde{P}$}
		\put (94,70) {$D_{l'}$}
	\end{overpic}
\end{minipage}
\caption{Left figure: The polygon $P'\subset P$. Middle figure: The region $D_l$ intersects both $\partial ([0,1]^2)$ and $R$, and the intersection $D_l\cap R$ is not contained in the strip $S$. Thus, we separate $D_l$ from $B$ by a tangent line. Right figure: The polygon $\widetilde P$. Note that there exists a region $D_{l'}$ intersecting $\partial ([0,1]^2)$ and $\widetilde P$. However, the intersection $D_{l'}\cap \widetilde P$ is contained in the strip $S$ and we can neglect it.}\label{figure:polygon_Ptilde}
\end{figure}

If $D_l$ is a region intersecting $\partial ([0,1]^2)$ and $P'$, then it has to intersect $\partial R$, since $P'\subset R\subset [0,1]^2$. Moreover, if $D_l\cap R$ is not entirely contained in the strip $S$, then it must intersect $\widetilde R$ and $\partial \widetilde R$, since $P' \subset \widetilde R\cup S$. Therefore, $$\diam(D_l\cap (R\setminus \widetilde R) )\geq \min\{ (k/2\pi)\lambda^2,\lambda\} \length(\partial B).$$
If $\lambda$ is sufficiently small, depending on $k$, then $(k/2\pi)\lambda^2 \leq \lambda$ and 
$$\diam(D_l\cap (\widetilde R\setminus R))\geq (k/2\pi)\lambda^2\length(\partial B).$$
Consider a ball $B(z,r)$ centered at some point $z\in \partial \widetilde R\cap D_l$ and with radius $r=(k/2\pi)\lambda^2\length(\partial B)$. By the Ahlfors regularity of $D_l$, we have 
\begin{align*}
\area(D_l\cap B(z,r)) \geq M\lambda^4 \length(\partial B)^2
\end{align*}
where $M>0$ and depends only on $k$. Note that the intersection $D_l\cap B(z,r)$ is contained in the $r$-neighborhood of $\widetilde R$, whose area is bounded above by 
$$\length(\partial \widetilde R)cr\leq \length(\partial R) cr \leq c'\diam(B)r \leq c'' \lambda^2 \length(\partial B)^2,$$
where $c''>0$ and depends only on $k$. We conclude that the number of regions $D_l$ that intersect $\partial ([0,1]^2)$ and $P'$ and are not contained in the strip $S$ is bounded above by $c'''\lambda^{-2}$, 
where $c'''>0$ and depends only on $k$.

We now construct a polygon $\widetilde P \subset P'$ that circumscribes $B$ by separating each of these regions $D_l$ from $B$ with a line tangent to $\partial B$; see Figure \ref{figure:polygon_Ptilde}. By construction, we have  
\begin{align}\label{theorem:zp_tilde}
\# Z(\widetilde P) \leq Z(P')+ c'''\lambda^{-2}\leq c\lambda^{-2}
\end{align}
where $c>0$ is a constant depending only on $k$. This completes the construction of the polygon $\widetilde P$. 

Summarizing, by construction, if $L$ is a vertical line with $L\cap B\neq \emptyset$ and $L\cap S=\emptyset$, where $S$ is the $\lambda\length(\partial B)$-strip of $Z(R)$, then $L\cap \widetilde P$ does not intersect any region $D_l$ with $D_l\cap \partial ([0,1]^2)\neq \emptyset$.

\bigskip

\subsection{Relative position between $B$ and the squares $K_j$}
In this final subsection we complete the proof of Theorem \ref{theorem:main}. We will study two main cases for the relative position and size of the region $B$ and of the squares $K_j$. We first establish an auxiliary lemma.

\begin{lemma}\label{lemma:midpoint}
Let $L=L_x$ be a vertical line with $L\cap B\neq \emptyset$ such that $L$ does not intersect the $\lambda\length(\partial B)$-strip of $Z(R)$, where $\lambda\in (0,1)$. Then there exists a constant $c_1>0$ depending only on $k$ and a point $z_1\in L\cap B$ such that
\begin{align*}
\dist(z_1,\partial B) \geq c_1\lambda^2 \length(\partial B).
\end{align*} 
\end{lemma}
In fact, the conclusion holds with $\lambda$ in the place of $\lambda^2$, but for our purposes this statement is enough.

\begin{proof}Without loss of generality, we suppose that $\length(\partial B)=1$. Let $L$ be a vertical line as in the statement. Consider the line $L'$ joining the leftmost and rightmost points $w_l$ and $w_r$ of $\partial B$, respectively. Note that $w_l$ and $w_r$ are points of tangency between $\partial B$ and $\partial R$ and hence they are contained in $Z(R)$. We claim that the point $z_1$ lying in the intersection of the lines $L$ and $L'$ satisfies the conclusion. By the strict convexity of $B$, the point $z_1$ lies necessarily in $B$.

Let $z_2$ be a point in $\partial B$ such that $\dist(z_1,\partial B)=|z_1-z_2|$ and consider the tangent line $L(z_2)$ of $\partial B$ at $z_2$. We note that $|z_1-z_2|=\dist(z_1,L(z_2))$. Since $z_1,w_l,w_r$ lie on the same line, it follows that the distance $\dist(z_1,L(z_2))$ is bounded below by either $\dist(w_l,L(z_2))$ or $\dist(w_r,L(z_2))$. Without loss of generality, we assume that $|z_1-z_2| \geq \dist(w_l, L(z_2))$. By Lemma \ref{lemma:arc_tangent} (i), there exists a constant $c>0$ depending only on $k$ such that $\dist(w_l,L(z_2)) \geq c \length(\partial B|[w_l,z_2])^2$.

If $z_2$ does not lie in the $(\lambda/2)$-strip of $w_l$, then $\length(\partial B|[w_l,z_2]) \geq  \lambda/2$. Therefore, $\dist(z_1,\partial B) \geq c\lambda^2/4$. If $z_2$ lies in the $(\lambda/2)$-strip of $w_l$, then $|z_1-z_2|\geq |\re(z_1-z_2)| \geq \lambda/2$, since $z_1$ does not lie in the $\lambda$-strip of $w_l$. Moreover, $\lambda\geq \lambda^2$ since $\lambda<1$. Summarizing, if we set $c_1= \min \{c/4, 1/2\}$, then we have $\dist(z_1,\partial B)\geq c_1\lambda^2$, as desired.
\end{proof}

Let $L=L_x$ be a line with $L\cap B\neq \emptyset$ such that $L$ does not intersect the set $Z(\widetilde P)$ or the $\lambda\length(\partial B)$-strip of the set $Z(R)$. Then $L\cap(\widetilde P\setminus B)$ has two components, $I_1$ and $I_2$.  By the construction of the polygon $\widetilde P$,  $I_1$ and $I_2$ do not intersect any region $D_l$ that intersects $\partial([0,1]^2)$, so the assumption (iii-2) of Lemma \ref{lemma:polygon_good_estimate} holds. Moreover, the assumption (iii-1) also holds since $\widetilde P\subset P$ and $P$ was constructed so that (iii-1) holds. In order to be able to apply the estimate of part (iii) from Lemma \ref{lemma:polygon_good_estimate} we need to ensure that $I_1$ and $I_2$ lie in distinct components of $L_x\cap (\bigcup_{j=1}^p K_j \cup \bigcup_{i={m+1}}^n B_i)$. 

If $I_1$ and $I_2$ lie in the same component of $L_x\cap (\bigcup_{j=1}^p K_j \cup \bigcup_{i={m+1}}^n B_i)$, then this component contains $L\cap B$, so $L\cap B\subset \bigcup_{j=1}^p K_j$. By Lemma \ref{lemma:midpoint} there exists a point $z_1\in L\cap B$ such that $\dist(z_1,\partial B) \geq c_1\lambda^2 \length(\partial B)$. In particular, there exists a square $K_j$ such that $z_1\in K_j$. Recall that, by assumption, no square is contained in any region $D_i$, since these squares do not contribute to the estimation of the Hausdorff dimension; see the comments before \eqref{theorem:sum_diameter_bound} in Section \ref{section:initial_setup}. Therefore, $K_j$ intersects $\partial B$. It follows that 
\begin{align*}
\diam(K_j) \geq c_1\lambda^2 \length(\partial B).
\end{align*}

We set $\lambda=s-1 \in (0,1)$ and choose $s$ sufficiently close to $1$, depending only on $k$, so that all previous claims in the construction of $\widetilde P$ hold for that value of $\lambda$. We now consider two cases; see Figure \ref{figure:cases}.

\begin{figure}
\centering
\begin{minipage}{0.49\linewidth}
\centering
\begin{overpic}[width=.9\linewidth]{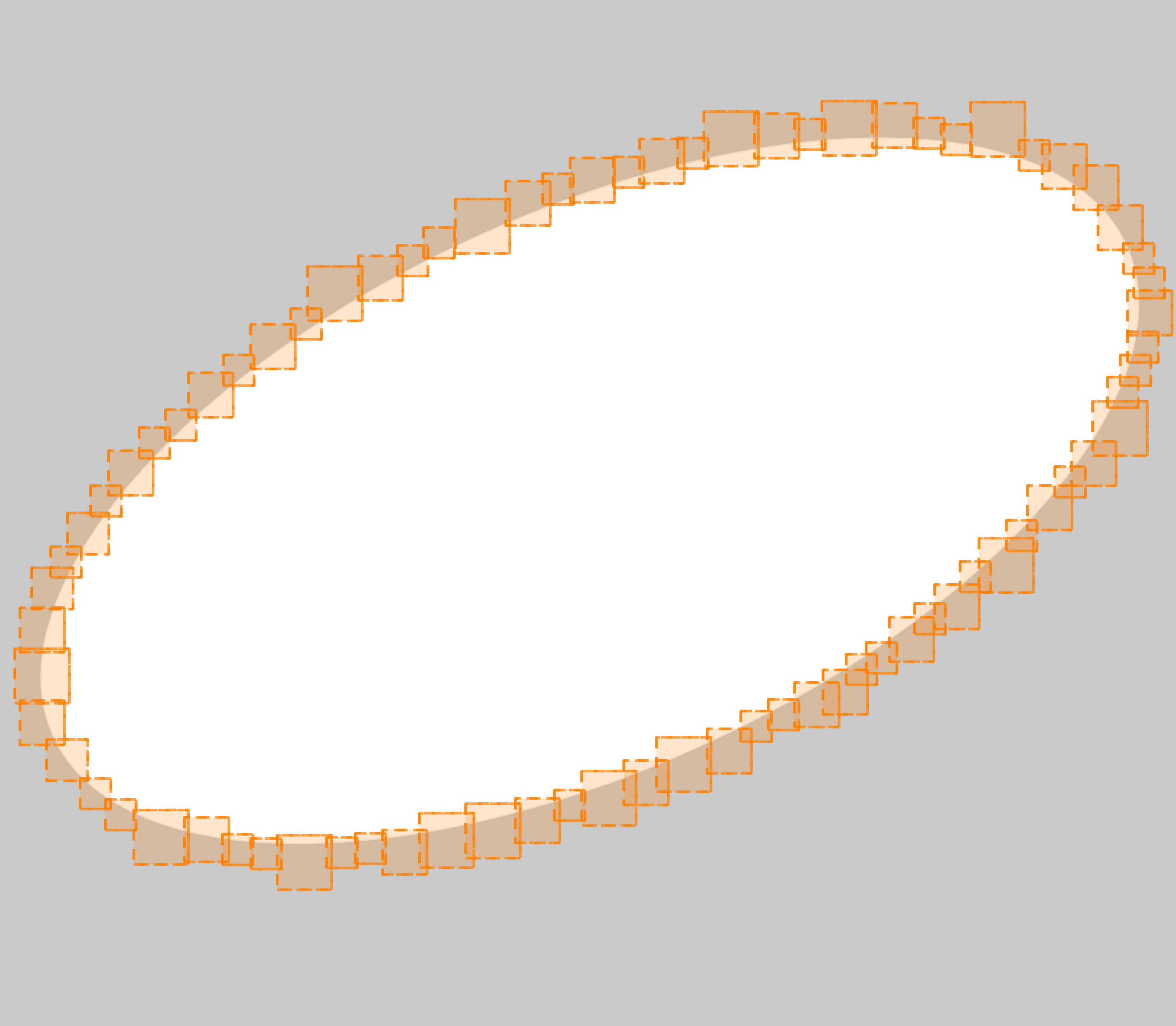}
		\put (47,43) {$B$}
	\end{overpic}
\end{minipage}
\begin{minipage}{0.49\linewidth}
\centering
	\begin{overpic}[width=.9\linewidth]{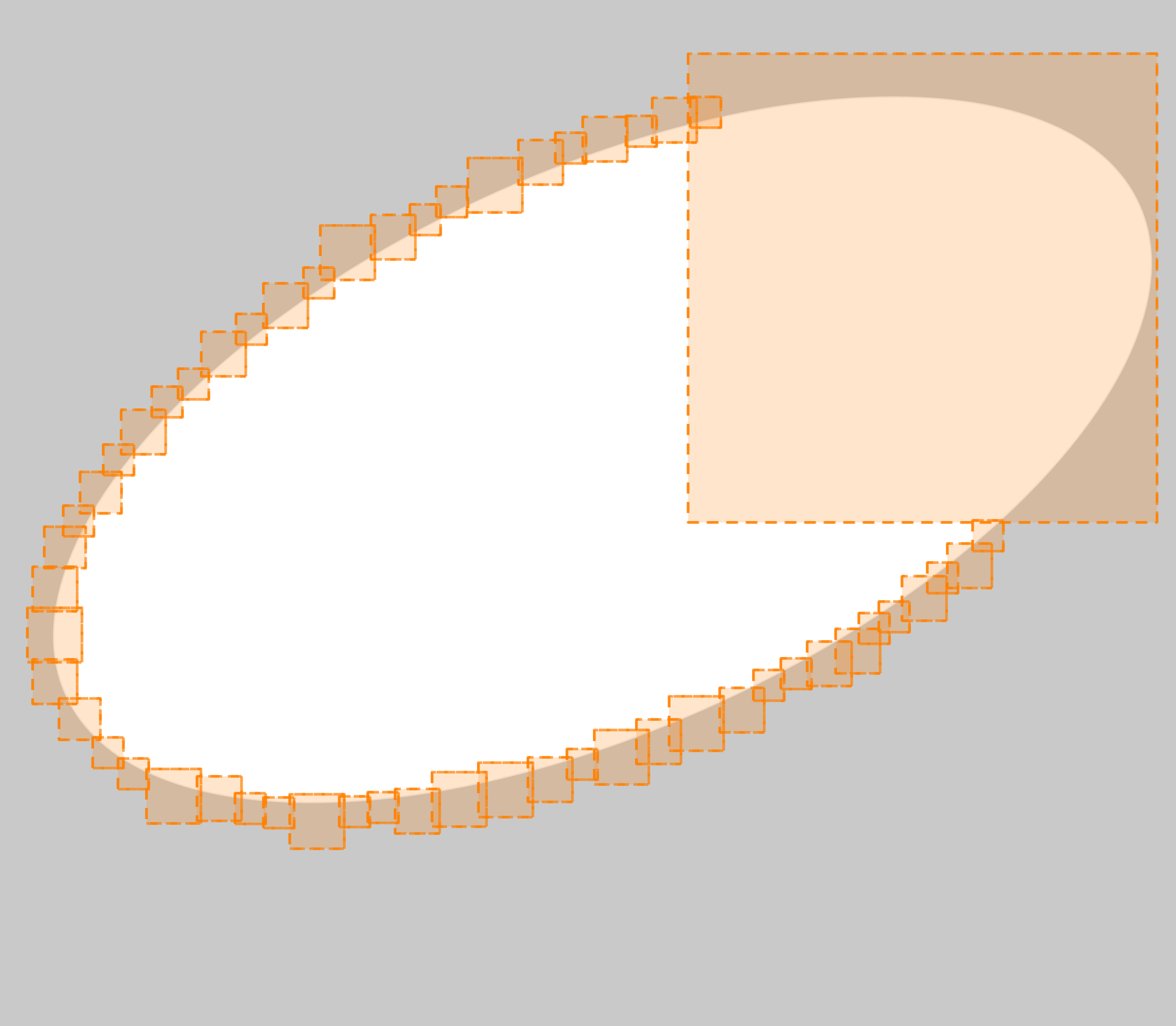}
		\put (47,43) {$B$}
		\put (76,60) {$K_j$}
	\end{overpic}
\end{minipage}
	\caption{Illustration of Case 1 on the left and Case 2 on the right. In Case 1 all squares intersecting $B$ are relatively small, while in Case 2 there exists a square $K_j$ with $\diam(K_j)\geq c_1\lambda^2\length(\partial B)$.}\label{figure:cases}
\end{figure}

\bigskip

\noindent
\textbf{Case 1.}
For each $j\in \{1,\dots,p\}$, if $K_j\cap B\neq \emptyset$, then $\diam(K_j) < c_1\lambda^2 \length(\partial B)$.

\bigskip
\noindent
By the previous comments, it follows that $I_1=I_1(x)$ and $I_2=I_2(x)$ lie in distinct components of $L_x\cap (\bigcup_{j=1}^p K_j \cup \bigcup_{i={m+1}}^n B_i)$, whenever $L_x$ is a line with $L_x\cap B\neq \emptyset$ and $L_x$ does not intersect $ Z(\widetilde P)$ or the $\lambda\length(\partial B)$-strip of $Z(R)$. Suppose, in addition, that  $L_x$ does not intersect the $\lambda^3\length(\partial B)$-strip of $ Z(\widetilde P)$. We denote by $\mathcal G$ the set of values of $x\in \proj(B)$ satisfying all these restrictions. If $x\in \mathcal G$, we can apply Lemma \ref{lemma:polygon_good_estimate} (iii) to the polygon $\widetilde P$ with $\lambda^3=(s-1)^3$ in the place of $\lambda$, and we obtain
\begin{align*}
g_m(x)-g_{m-1}(x) \geq \frac{1}{2} \length(\partial B)^{s-1}.
\end{align*}
for $s\in (1,s_0)$, where $s_0$ depends only on $k$. The set $\mathcal B=\proj(B)\setminus \mathcal G$, where the above estimate might fail, has measure at most
$$|\mathcal B|\leq  2\lambda\length(\partial B)\cdot \# Z(R) + 2\lambda^3\length(\partial B) \cdot \# Z(\widetilde P) \leq c\lambda \length(\partial B)$$
by \eqref{theorem:zp_tilde}, where $c>0$ is a constant depending only on $k$. When $x\in \mathcal B$, we use the crude estimate from Lemma \ref{lemma:basic_estimate} (ii):
$$g_m(x)-g_{m-1}(x) \geq -\length(L_x\cap B)^{s-1} \geq -\length(\partial B)^{s-1}.$$ 
By Lemma \ref{lemma:arc_tangent} (ii) we have {$|\proj(B)|\geq c'\length(\partial B)$}, where $c'$ depends only on $k$, thus
$$|\mathcal G| \geq (c'-c\lambda)\length(\partial B) \geq c'' \length(\partial B),$$
as soon as $\lambda$ is sufficiently small, depending only on $k$. Therefore,
\begin{align*}
\int_0^1(g_m-g_{m-1}) &= \int_{\mathcal G}(g_m-g_{m-1}) +\int_{\mathcal B}(g_m-g_{m-1})\\
&\geq (c''/2) \length(\partial B)^{s}- c\lambda \length(\partial B)^{s}\\
&=(c''/2-c\lambda) \length(\partial B)^{s}>0, 
\end{align*}
provided $\lambda=s-1$ is sufficiently small, depending on $k$. This completes the proof of \eqref{theorem:main_estimate} in this case. 

\bigskip
\noindent
\textbf{Case 2.} There exists $j\in \{1,\dots,p\}$ (depending on $B$) such that $K_j\cap B\neq \emptyset$ and $\diam(K_j)\geq c_1\lambda^2 \length(\partial B)$. 

\bigskip
\noindent
We note that if $B\cap \partial K_j=\emptyset$, then $B$ is entirely contained in $K_j$. In this case for each vertical line $L=L_x$ intersecting $B$ the segments $I(1)$ and $I(w)$, with the notation of Lemma \ref{lemma:basic_estimate}, coincide and we have $g_m(x)-g_{m-1}(x)=0$ for all $x\in \proj(B)$. Therefore,
\begin{align*}
\int_0^1 (g_m-g_{m-1})=0.
\end{align*}   
This proves the desired main estimate \eqref{theorem:main_estimate}. 

We now suppose that $B\cap \partial K_j \neq \emptyset$. Using the crude estimate from Lemma \ref{lemma:basic_estimate} (ii) we have
\begin{align*}
g_m(x)-g_{m-1}(x)\geq -\length(\partial B)^{s-1},
\end{align*}
so 
\begin{align*}
\int_0^1 (g_m-g_{m-1}) &\geq -\length(\partial B)^{s-1} |\proj (B)| \geq -\length(\partial B)^s\geq -c_2^{-s} \diam(B)^s,
\end{align*}
where the last inequality follows from the chord-arc property of $\partial B$ and $c_2>0$ is a constant depending only on $k$. This proves the desired main estimate \eqref{theorem:main_estimate} in this case too.  The proof of the main theorem has been completed. \qed

\bigskip

\section{Proof of Theorem \ref{theorem:relaxation}}\label{section:relaxation}

We only need to verify that all auxiliary lemmas from Section \ref{section:preliminaries} that were used in the proof of Theorem \ref{theorem:main} also hold under the less restrictive assumption of Theorem \ref{theorem:relaxation}. Moreover, throughout the entire proof of Theorem \ref{theorem:main}, instead of using tangent lines of a smooth and convex curve $C$, one has to use \textit{supporting lines}, i.e., lines intersecting $C$ and separating it from a half-plane.

For two compact sets $A,B\subset \C$ we define the \textit{Hausdorff metric} $d_H(A,B)$ to be the infimum of all $\varepsilon>0$ such that $A$ is contained in the open $\varepsilon$-neighborhood of $B$ and $B$ is contained in the open $\varepsilon$-neighborhood of $A$. We say that a sequence of compact sets $C_n$ converges to a compact set $C$ in the {Hausdorff metric} if $d_H(C,C_n)\to 0$ as $n\to\infty$. 

First we prove an auxiliary lemma.

\begin{lemma}\label{lemma:length_convergence}
Let $C_n$, $n\in \N$, be a sequence of planar convex Jordan curves converging in the Hausdorff metric to a Jordan curve $C$. Then $C$ is convex and the length of $C_n$ converges to the length of $C$. Moreover, for each arc $A\subset C$ with endpoints $z,w$ there exists a sequence of arcs $A_n\subset C_n$ with endpoints $z_n,w_n$, $n\in \N$, such that $z_n\to z$, $w_n\to w$, $A_n\to A$ in the Hausdorff metric, and $\length(A_n)\to \length(A)$ as $n\to\infty$.
\end{lemma}
\begin{proof}
Convergence in the Hausdorff metric implies that for each $z\in C$ there exists a sequence $z_n\in C_n$ converging to $C$. A consequence of this observation is that $C$ has to be convex.  We denote by $D$ the Jordan region bounded by $C$ and by $D_n$ the region bounded by $C_n$.

Let $z_0\in D$. For $\varepsilon>0$ consider the convex Jordan curve $C(\varepsilon)= \{z_0+(1+\varepsilon)(z-z_0): z\in C \}$. By the convexity of $C$, the curve $C( \varepsilon)$ is disjoint from $C$ and $C$ is contained in the interior region of $C(\varepsilon)$, denoted by $D(\varepsilon)$. Moreover, $\length(C( \varepsilon))=(1+\varepsilon)\length(C)$. By compactness, for each $\varepsilon>0$ there exists $\delta>0$ such that $D(\varepsilon)$ contains the $\delta$-neighborhood of $C$. In particular, 
\begin{align}\label{hausdorff:containment1}
C_n\subset D(\varepsilon)
\end{align}
for all sufficiently large $n$. 

Let $B(z_0,r_0)$ be a ball such that $B(z_0,2r_0)$ is contained in $D$. We claim that the ball  $B(z_0,r_0)$ is contained in $D_n$ for all sufficiently large $n$. Suppose that this is not the case. Since $C_n$ converge to $C$, it has to be disjoint from $B(z_0,r_0)$ for all sufficiently large $n$. If $B(z_0,r_0)$ is not contained $D_n$, then there exists a line separating $D_n$ from a half-plane containing $B(z_0,r_0)$. In this case, no point of $C_n$ can approach points of $C$ that lie on that half-plane.  Hence, $B(z_0,r_0)$ is not contained in $D_n$ for at most finitely many $n$. The convergence of $C_n$ to $C$ also implies that there exists a large ball $B(z_0,R_0)$ that contains $D$ and $D_n$ for all $n\in \N$. Summarizing, we have
\begin{align*}
B(z_0,r_0) \subset D,D_n\subset B(z_0,R_0)
\end{align*}
for all $n\in \N$. According to \cite[Lemma 4.1, Claim 1]{BishopDrillickNtalampekos:Falconer}, this implies that there exists a constant $M=M(r_0,R_0)>0$ such that if a ray emanating from $z_0$ hits $C$ and $C_n$ at points $z$ and $w$, respectively, then
\begin{align}\label{hausdorff:inequality}
|z-w|\leq M d_H(C,C_n).
\end{align}

For $\varepsilon>0$ we consider the curve $C_n(\varepsilon)=\{z_0+(1+\varepsilon)(z-z_0): z\in C_n \}$ and denote by $D_n(\varepsilon)$ the Jordan region bounded by $C_n(\varepsilon)$. By \eqref{hausdorff:inequality}, if $Md_H(C,C_n)<\varepsilon$, then any ray emanating from $z_0$ hits first $C$ and then $C_n(\varepsilon)$. Hence, we conclude that 
\begin{align}\label{hausdorff:containment2}
C\subset D_n(\varepsilon)
\end{align}
for all sufficiently large $n$.

Summarizing, by \eqref{hausdorff:containment1} and \eqref{hausdorff:containment2}, for each $\varepsilon>0$ and for all sufficiently large $n$ we have $C_n\subset D(\varepsilon)$ and $C\subset D_n(\varepsilon)$. An application of Crofton's formula (see \cite[p.\ 38]{Tabachnikov:geometry}) shows that 
$$\length(C_n)\leq \length(C(\varepsilon)) \quad \textrm{and} \quad \length(C)\leq \length(C_n(\varepsilon))$$
for all sufficiently large $n$. Letting first $n\to\infty$ and then $\varepsilon\to 0$ shows that $\length(C_n)$ converges to $\length(C)$, as desired. This proves the first part of the lemma.

We now parametrize $C$ and $C_n$ using rays emanating from $z_0$. That is, for each $a\in \partial B(z_0,r_0)$ we let $\gamma(a)$ (resp.\ $\gamma_n(a)$) be the unique point of intersection of $C$ (resp.\ $C_n$) with the ray $z_0+\lambda(a-z_0)$, $\lambda>0$. By \eqref{hausdorff:inequality}, it follows that $\gamma_n$ converges to $\gamma$ uniformly. Let $A\subset C$ be an arc with endpoints $z=\gamma(a)$ and $w=\gamma(b)$. We consider the arcs $A_n= \gamma_n(\gamma^{-1}(A))$ with endpoints $z_n=\gamma_n(\gamma^{-1}(z))$ and $w_n=\gamma_n(\gamma^{-1}(w))$, $n\in \N$. By the uniform convergence, it follows that $A_n\to A$ in the Hausdorff metric.

Finally, we show that $\length(A_n)\to \length(A)$. We note that we have
\begin{align*}
\length(A)&\leq \liminf_{n\to\infty}\length(A_n) \quad \textrm{and}\\
\length(C\setminus A)&\leq \liminf_{n\to\infty}\length(C_n\setminus A_n).
\end{align*}
Since $\lim_{n\to\infty}\length(C_n)=\length(C)$, it follows that $\length(A_n)\to \length(A)$, as desired.
\end{proof}

Next, we start verifying the auxiliary results of Section \ref{section:preliminaries}.  We first prove the strict convexity of $C$. Suppose, for the sake of contradiction, that $C$ contains a line segment $A$ in its boundary. Then, by Lemma \ref{lemma:length_convergence}, there exists a sequence of arcs $A_n\subset C_n$ converging to $A$. The ``thickness'' of the arc $A_n$ can be measured by $$\inf_{w\in A_n} \sup_{z\in A_n}\dist(z,L(w)),$$ where $L(w)$ is a line tangent to $A_n$ at the point $w$. By Lemma \ref{lemma:arc_tangent} (i), the thickness of $A_n$ is at least $c\length(A_n)^2/\length(C_n)$, where $c$ is a constant depending only on $k$. Since $\length(A_n)\to \length(A)$ and $\length(C_n)\to \length(C)$, it follows that the thickness of $A_n$ is uniformly bounded below as $n\to\infty$. However, since $A_n$ converges to a straight line segment, the thickness converges to $0$, a contradiction. This establishes the strict convexity of $C$. Therefore, Lemma \ref{lemma:intersection_of_two} holds for curves $\partial D_1$ and $\partial D_2$ that can be approximated by such curves $C_n$.

Next, we discuss the validity of Lemma \ref{lemma:arc_tangent}. Note that part (ii) holds immediately from the  convergence in the Hausdorff metric and from Lemma \ref{lemma:length_convergence}. Part (i) holds if $L(z_2)$ is a {supporting line}, rather than a tangent line. Indeed, we let $z_1,z_2\in C$ and $L(z_2)$ be a line that intersects $C$ only at $z_2$. There exists another supporting line $L'(z_2)$ (we could have $L(z_2)=L'(z_2)$) with extremal slope such that $\dist(z_1,L(z_2))\geq \dist(z_1,L'(z_2))$; {see Figure \ref{figure:supporting}. It suffices to show that $$\dist(z_1,L'(z_2)) \geq \frac{k}{2\pi} \frac{\length(C|[z_1,z_2])^2}{\length(C)}.$$

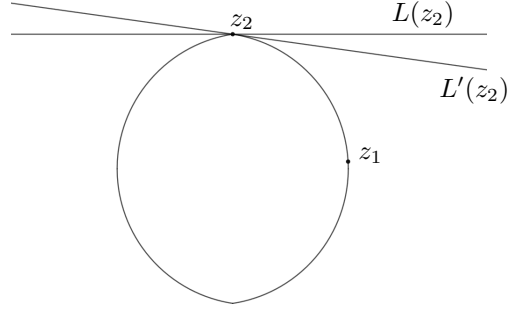
\begin{figure}
	\begin{overpic}[width=.5\linewidth]{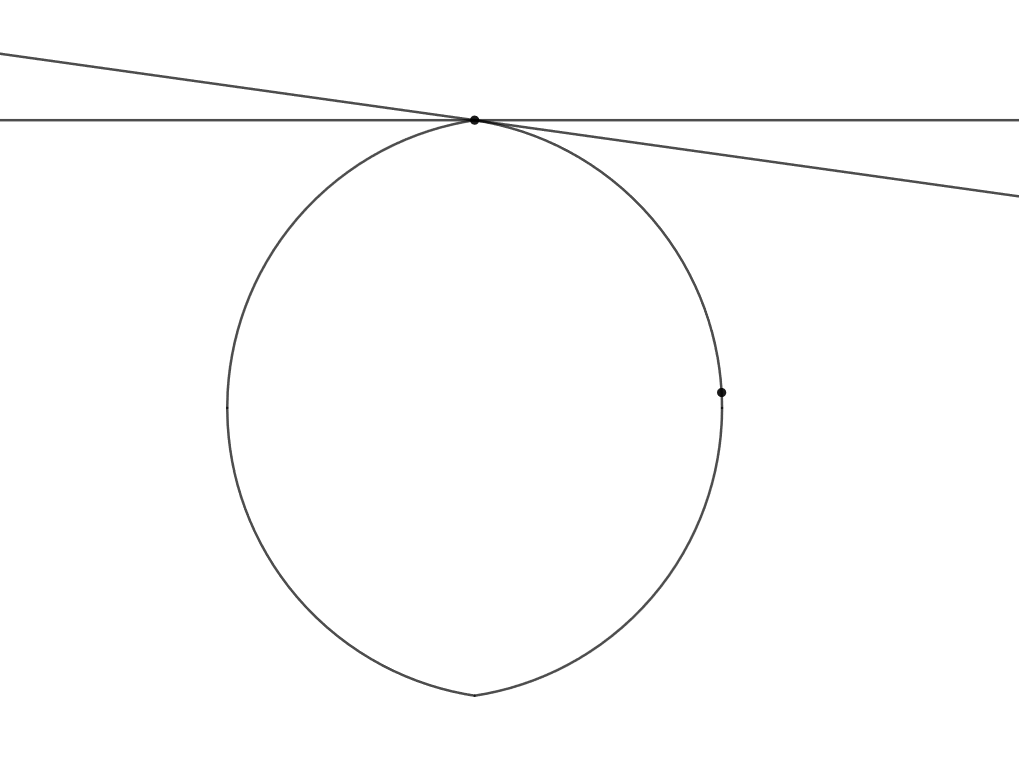}
		\put (73, 38) {$z_1$}
		\put (46, 66) {$z_2$}
		\put (80, 67) {$L(z_2)$}
		\put (90, 51) {$L^\prime(z_2)$}
	\end{overpic}
	\caption{The curve $C$ with two supporting lines $L(z_2),L^\prime(z_2)$ passing through $z_2$.}\label{figure:supporting}
\end{figure}

We can find points $z_{1,n}, z_{2,n} \in C_n$ converging to $z_1,z_2$, respectively, such that the tangent lines of $C_n$ at $z_{2,n}$ converge to the line $L'(z_2)$; this is not necessarily true for the original line $L(z_2)$. The convergence of tangents, by working locally (if one uses as local coordinates the projections from $C$ and $C_n$ to $L'(z_2)$), follows from the following lemma about convex functions, which we will prove later.

\begin{lemma}\label{lemma:convex}
Let $f_n\colon (a,b)\to \R$, $n\in \N$, be a sequence of convex, smooth functions and suppose that $f_n$ converges pointwise to a function $f\colon (a,b)\to \R$. Then for each $c\in (a,b)$ and for each $\varepsilon>0$ there exist $\delta>0$ and $N\in \N$ such that for $n\geq N$ we have 
\begin{align*}
|f_n'(c \pm \delta) -f'_{\pm}(c) | <\varepsilon,
\end{align*}
where $f'_{\pm}(c)$ denote the one-sided derivatives of $f$. 
\end{lemma}

Lemma \ref{lemma:arc_tangent} (i) applies now to the curve $C_n$  and to the tangent $L(z_{2,n})$ at $z_{2,n}$. Since the length of $C_n|[z_{1,n},z_{2,n}]$ converges to the length of $C|[z_1,z_2]$ and the length of $C_n$ converges to the length of $C$, the conclusion follows. 

Theorem \ref{theorem:chordarc} holds immediately for $C$ by Lemma \ref{lemma:length_convergence}. Therefore, $C$ is a chord-arc curve. Lemma \ref{lemma:chordarc_area} implies that the region bounded by $C$ is Ahlfors $2$-regular, so all of Section \ref{section:ahlfors} is valid. \qed

\begin{proof}[Proof of Lemma \ref{lemma:convex}]
Pointwise convergence implies that $f$ is convex on $(a,b)$. The convexity of $f$ implies that $f$ is continuous, and for all  $c\in (a,b)$ the one-sided derivatives $f'_{\pm}(c)$ exist and they are monotonically increasing. Hence, $f'(c)$ exists for all but countably many $c\in (a,b)$. Moreover, we have
$$\lim_{x\to c^{\pm}} f'_{\pm}(x) =f'_{\pm}(c)$$
for all $c\in (a,b)$; see \cite[Theorem 24.1, p.~227]{Rockafellar:convex}. It follows that that for each $c\in (a,b)$ and $\varepsilon>0$  there exists $\delta>0$ such that $f'(c\pm \delta)$ exists and $|f'(c\pm \delta)-f'_{\pm}(c)|<\varepsilon$. By \cite[Theorem 24.5, p.~233]{Rockafellar:convex}, if $f'(x)$ exists, then $\lim_{n\to\infty}f_n'(x)=f'(x)$. Therefore, for all sufficiently large $n\in \N$ we have $|f_n'(c\pm \delta)-f'_{\pm}(c)|<\varepsilon$, as desired.
\end{proof}
\bigskip

\section{Proof of Theorems \ref{theorem:counterexample} and \ref{theorem:dimension_one}}\label{section:example}
The proofs of parts (i) and (ii) of Theorem \ref{theorem:counterexample} are similar. We first provide the proof of part (i) and then we will explain the modifications needed in order to obtain (ii).

\begin{proof}[Proof of Theorem \ref{theorem:counterexample} (i)]
We first construct a family of convex, smooth Jordan curves with certain properties that will be used in the construction of the packing $\mathcal P_{\Omega}$.  For each square $Q=[a,b]\times [c,d]$ and for each $\alpha>0$ consider a  convex, smooth Jordan curve $C_{\alpha,Q}$ that agrees with $\partial Q$, except at four balls of radius $\alpha$, centered at the vertices of the square $Q$. We denote by $D_{\alpha,Q}$ the Jordan region bounded by $C_{\alpha,Q}$; see Figure \ref{figure:D_alpha_1}. By construction, $Q\setminus \br{D_{\alpha,Q}}$ is contained in four balls of radius $\alpha$. 

\begin{figure}
\centering
\begin{minipage}{0.49\textwidth}
\centering
	\captionsetup{width=.8\linewidth}
	\begin{overpic}[width=1\linewidth]{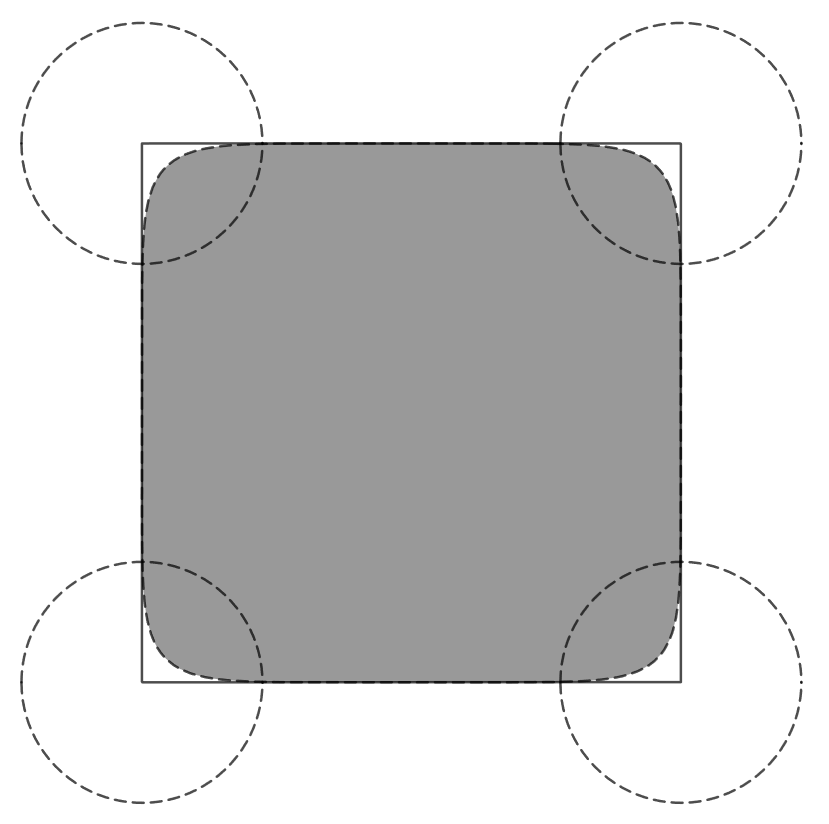}
		\put (85, 50) {$Q$}
		\put (50, 50) {$D_{\alpha,Q}$}
	\end{overpic}
	\caption{The convex region $D_{\alpha,Q}$ in the proof of Theorem \ref{theorem:counterexample} (i).}\label{figure:D_alpha_1}
\end{minipage}
\begin{minipage}{0.49\textwidth}
\centering
	\captionsetup{width=.8\linewidth}
	\begin{overpic}[width=1\linewidth]{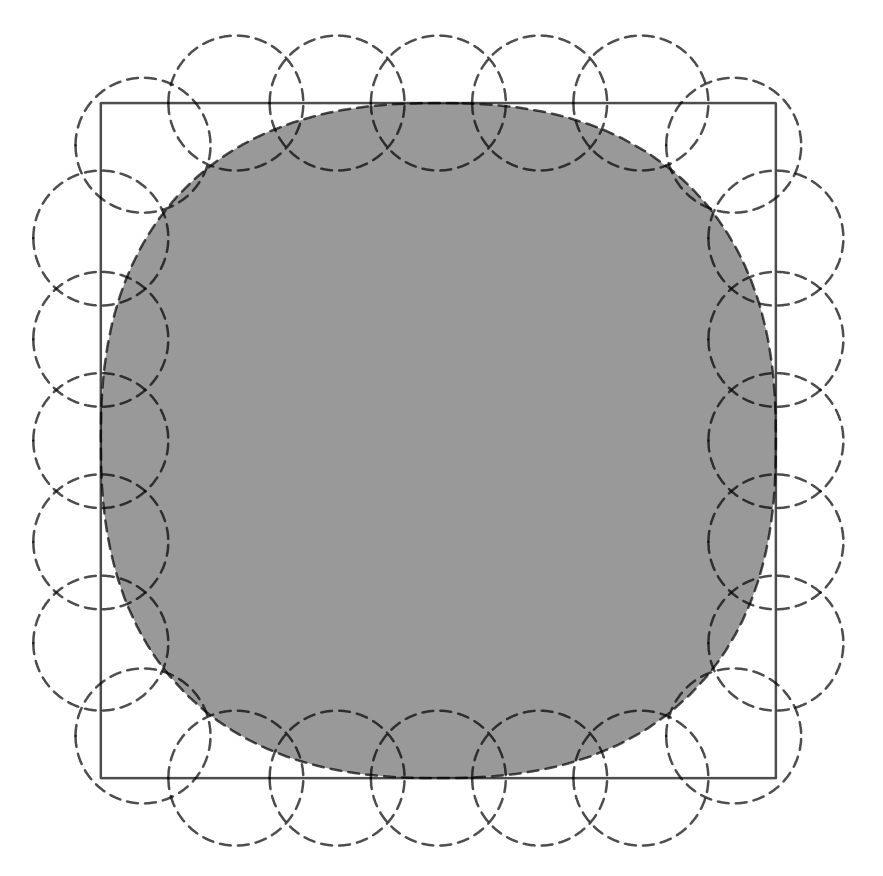}
		\put (96, 45) {$Q$}
		\put (50, 50) {$D_{\alpha,Q}$}
	\end{overpic}
	\caption{The strictly convex region $D_{\alpha,Q}$ in the proof of Theorem \ref{theorem:counterexample} (ii).}\label{figure:D_alpha_2}
\end{minipage}
\end{figure}

Consider a sequence $\varepsilon_n$, $n\in \N$, with $\varepsilon_n\to 0$ as $n\to\infty$. Let $\Omega\subset \R^2$ be an open set and consider a cover of $\Omega$ by closed squares $Q_i\subset \Omega$, $i\in \N$, with disjoint interiors such that each compact subset of $\Omega$ intersects only finitely many squares; see \cite[Theorem 1.4]{SteinShakarchi:RealAnalysis}. For each square $Q_i$ we consider a region $D_{1,i}=D_{\alpha_i,Q_i}$, for $\alpha_i=\varepsilon_1/2^i$, such that $Q_i\setminus \br{D_{1,i}}$ is contained in four balls of radius $\varepsilon_1/2^i$, centered at the four corners of $Q_i$.  Then $\Omega\setminus \bigcup_{i=1}^\infty \br {D_{1,i}}$ is covered by balls $B_{1,j}$ of radius $r_{1,j}<\varepsilon_1$, $j\in \N$, with 
\begin{align*}
\sum_{j=1}^\infty r_{1,j} = 4 \varepsilon_1.
\end{align*}
We note that $\bigcup_{i=1}^\infty \br {D_{1,i}}$ is a closed subset of $\Omega$. Indeed, if $z_i\in \br {D_{1,i}}$, $i\in \N$, then $z_i\in Q_i$. Since each compact subset of $\Omega$ intersects only finitely many squares $Q_i$, it follows that the sequence $\{z_i\}_{i\in \N}$  accumulates only at $\partial \Omega$. We define $\Omega_1=\Omega\setminus \bigcup_{i=1}^\infty \br {D_{1,i}}$ and it follows that this is an open set. Hence, we can repeat the above procedure with $\Omega_1$ in the place of $\Omega$.

Inductively, once $\Omega_n$ has been defined, we obtain regions $D_{n+1,i}\subset \Omega_n$, $i\in \N$, such that the open set $\Omega_{n+1}\coloneqq \Omega_n\setminus \bigcup_{i=1}^\infty \br{D_{n+1,i}}$ is covered by balls $B_{n+1,j}$ of radius $r_{n+1,j}<\varepsilon_{n+1}$, $j\in \N$, with
\begin{align*}
\sum_{j=1}^\infty r_{n+1,j}= 4 \varepsilon_{n+1}.
\end{align*} 

We consider the packing $\mathcal P= \{D_{n,i}\}_{i,n\in \N}$ in $\Omega$. Its residual set $\mathcal S$ can be written as
\begin{align*}
\mathcal S= \left(\bigcup_{n,i\in \N} \partial D_{n,i}  \right)\cup \left( \Omega \setminus \bigcup_{n,i\in \N}\br{D_{n,i}} \right).
\end{align*}
Note that each curve $\partial D_{n,i}$ has finite length and thus finite Hausdorff $1$-measure. Moreover, by construction, for each $\varepsilon>0$ there exists a cover of the set  $E=\Omega \setminus \bigcup_{n,i\in \N}\br{D_{n,i}}$ by balls $B_j$, $j\in \N$, of radius $r_j<\varepsilon$ so that 
\begin{align*}
\sum_{j=1}^\infty \diam(B_j)=\sum_{j=1}^\infty 2r_j \leq 8\varepsilon.
\end{align*}
This implies that $\mathscr{H}_\varepsilon^{1}(E)\leq 8\varepsilon$. Letting $\varepsilon\to 0$ gives $\mathscr{H}^1(E)=0$. Summarizing, $\mathcal S$ is the countable union of sets of finite Hausdorff $1$-measure. In particular $\dim_{\mathscr{H}}(S)=1$.
\end{proof}

\begin{proof}[Proof of Theorem \ref{theorem:counterexample} (ii)]
The main modification we need to make in the proof of part (i) is in the family of curves $C_{\alpha,Q}$.  For each square $Q=[a,b]\times [c,d]$ and for each $0<\alpha<\ell(Q)$ consider a  strictly convex, smooth Jordan curve $C_{\alpha,Q}$ that bounds a Jordan region $D_{\alpha,Q}\subset Q$ such that $Q\setminus \br{D_{\alpha,Q}}$ is covered by $N=4\ell(Q)/\alpha$ balls of radius $\alpha$; recall that $\ell(Q)$ denotes the side length of $Q$. See Figure \ref{figure:D_alpha_2} for an illustration. The curve $C_{\alpha,Q}$ can be taken, for example, to be the image of the circle $|z|=r$, where $r<1$ is close to $1$, under a conformal map from the unit disk onto the interior of $Q$; since $Q$ is convex, such curves are strictly convex by Study's Theorem \cite{Study:theorem}.

We fix a sequence $s_n>1$ with $s_n\to 1$ and a sequence $\varepsilon_n>0$ with $\varepsilon_n\to 0$ as $n\to\infty$. As in the proof of (i), we write $\Omega=\bigcup_{i\in \N}Q_i$, where $Q_i$ are closed squares with disjoint interiors that accumulate only at $\partial \Omega$. For each square $Q_i$ we consider a region $D_{1,i}=D_{\alpha_i,Q_i}$, where $\alpha_i$ is chosen so that $4\ell(Q)\alpha_i^{s_1-1}<\varepsilon_1/2^i$ and $\alpha_i<\varepsilon_1$. Note that $Q_i\setminus \br{D_{1,i}}$ is covered by $4\ell(Q_i)\alpha_i^{-1}$ balls of radius $\alpha_i$. The set $\Omega_1=\Omega\setminus \bigcup_{i=1}^\infty \br {D_{1,i}}$ is open and is covered by balls $B_{1,j}$ of radius $r_{1,j}$, $j\in \N$, with
\begin{align*}
\sum_{j=1}^\infty r_{1,j}^{s_1}= \sum_{i=1}^\infty 4\ell(Q_i)\alpha_i^{s_1-1} <\sum_{i=1}^\infty \varepsilon_1 2^{-i} =\varepsilon_1.
\end{align*}
Inductively, one can define the open set $\Omega_{n+1}=\Omega_n\setminus \bigcup_{i=1}^\infty \br{D_{n+1,i}}\subset \Omega_n$ that is covered by balls $B_{n+1,j}$ of radius $r_{n+1,j}<\varepsilon_{n+1}$, $j\in \N$, with
\begin{align*}
\sum_{j=1}^\infty r_{n+1,j}^{s_{n+1}}<\varepsilon_{n+1}.
\end{align*}

We consider the packing $\mathcal P=\{D_{n,i}\}_{n,i\in \N}$. Its residual set can be decomposed into the union of the curves $\partial D_{n,i}$, $n,i\in \N$, with the set $E=\Omega \setminus \bigcup_{n,i\in \N}\br{D_{n,i}}$. It suffices to show that the set $E$ has Hausdorff dimension equal to $1$. Let $s>1$. By construction, for each $n\in \N$ there exists a cover of $E$ by balls $B_{n,j}$, $j\in \N$, of radius $r_{n,j}<\varepsilon_n$ so that 
\begin{align*}
\sum_{j=1}^\infty r_{n,j}^{s_n} <\varepsilon_n.
\end{align*}
If $n$ is sufficiently large, then $s_n<s$ and $r_{n,j}<\varepsilon_n<1$, so $r_{n,j}^s \leq r_{n,j}^{s_n}$. Therefore,
\begin{align*}
\sum_{j=1}^\infty \diam(B_{n,j})^s =
\sum_{j=1}^\infty  2^s r_{n,j}^{s} <2^s\varepsilon_n
\end{align*}
for sufficiently large $n$. This implies that $\mathscr{H}_{\varepsilon_n}^{s}(E)<2^s\varepsilon_n$. As $\varepsilon_n\to 0$, we obtain $\mathscr{H}^{s}(E)=0$. This holds for all $s>1$, hence, $\dim_{\mathscr{H}}(E)=1$.
\end{proof}

\begin{proof}[Proof of Theorem \ref{theorem:dimension_one}]
The proof is based on the fact that if a set $E$ has $\sigma$-finite Hausdorff $1$-measure, then for almost every horizontal line $L$ (with respect to Lebesgue measure) the intersection $L\cap E$ is at most countable; see \cite[Theorem 30.16, p.~104]{Vaisala:quasiconformal}. 

Consider a packing $\mathcal P_{\Omega}=\{D_i\}_{i\in \N}$ such that $\partial D_i\cap \partial D_j$ is at most countable for all $i\neq j$. Hence, the set
$$F=\bigcup_{\substack{i,j\in \N\\i\neq j}}(\partial D_i\cap \partial D_j)$$
is at most countable. 

For each $i\in \N$ we consider a parametrization $\gamma_i \colon \T\to \partial D_i$ and we let $G_i$ be the set of local maximum and local minimum values of the function $\im(\gamma_i) \colon \T\to \R$. Note that $y\in G_i$, if and only if there exists an open segment $I_y$ of the horizontal line $L_y=\{(x,y):x\in \R\}$ such that $I_y\cap \partial D_i\neq \emptyset$ and either ${I_y} \subset \br{D_i}$  or $I_y\subset \R^2 \setminus {D_i}$; see Figure \ref{figure:extrema}. The set of local extremal values of a real-valued function on a separable metric space is always at most countable; see \cite[Lemma 2.10]{Ntalampekos:monotone} for an argument. Therefore, the set $G_i$ is at most countable.

\begin{figure}
\centering
\begin{minipage}{0.49\textwidth}
\centering
\captionsetup{width=.8\linewidth}
\begin{overpic}[width=1\linewidth]{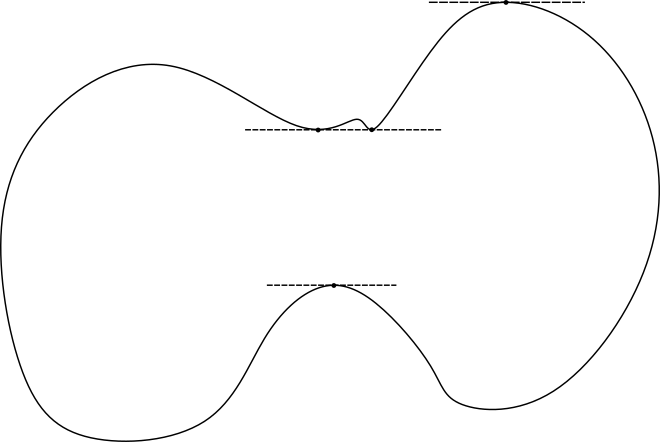}
		\put (90,10) {$D_i$}
		\put (61, 23) {$I_{y_1}$}
		\put (68, 47) {$I_{y_2}$}
		\put (90, 66) {$I_{y_3}$}
	\end{overpic}
	\caption{The values $y_1$ and $y_3$ are local maximum values of $\im(\gamma_i)$ and $y_2$ is a local minimum of $\im(\gamma_i)$. Each of the segments $I_{y_1},I_{y_2},I_{y_3}$ is contained either in $\br{D_i}$ or in $\R^2\setminus D_i$.}\label{figure:extrema}
\end{minipage}
\begin{minipage}{0.49\textwidth}
		\captionsetup{width=.8\linewidth}
		\begin{overpic}[width=1.\linewidth]{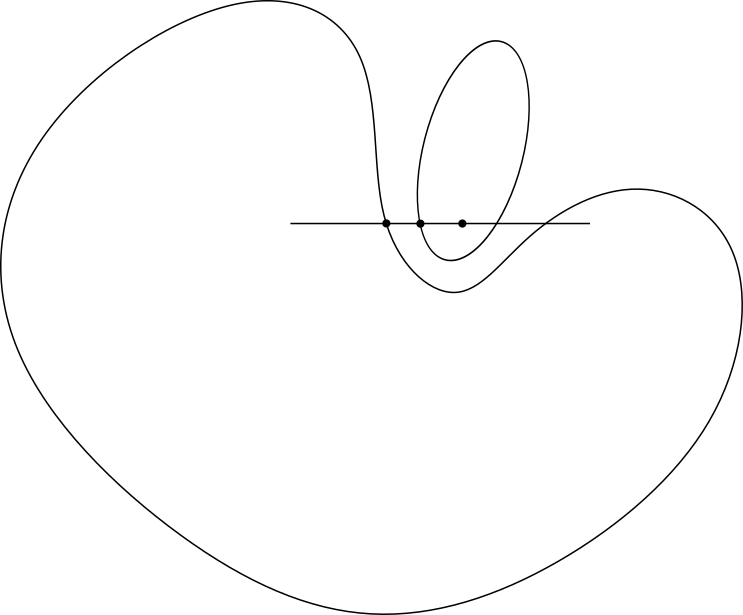}
		\put (30, 75) {$D_{i_0}$}
		\put (71, 77) {$D_j$}
		\put (35, 53) {$I$}
		\put (46, 54) {$z$}
		\put (57, 54) {$\tilde{w}$}
		\put (63, 54) {$w$}
		\end{overpic}
		\caption{The function $\im(\gamma_{i_0})$ does not attain a local extremal value at $z$, so the horizontal segment $I$ intersects both the interior and exterior of $D_{i_0}$. Moreover $z\notin \partial D_j$ for any $j\neq i_0$. }\label{figure:nonconvex}
\end{minipage}
\end{figure}

We will show that if a horizontal line $L$ does not intersect the countable set $F\cup \bigcup_{i\in \N}G_i$, then  the intersection of $L$ with the residual set $\mathcal S$ is uncountable. This will complete the proof.  In fact, we will show that the locally compact set $L\cap \mathcal S$ is perfect (in its relative topology). This will imply that $L\cap \mathcal S$ is uncountable.

Let $z\in L\cap \mathcal S = L\cap (\Omega \setminus \bigcup_{i=1}^\infty {D_i})$. Our goal is to show that $z$ is not isolated in $L\cap \mathcal S$, i.e., any open segment $I\subset L$ containing $z$ intersects $\mathcal S\setminus \{z\}$. Suppose first that $z\notin \partial D_i$ for any $i\in \N$. Then any open segment $I\subset L$ that contains $z$ is either contained in $\mathcal S$ or it intersects $ D_i$ for some $i\in \N$. In the second case, since $z\notin \br {D_i}$, there exists a point $w\in I\cap \partial D_i\subset L\cap \mathcal S$. In any case $(I\setminus \{z\})\cap \mathcal S\neq \emptyset$. 

Now, suppose that $z\in \partial D_{i_0}$ for some $i_0\in \N$ and let $I\subset L$ be an open segment that contains $z$. Since $I$ does not intersect the set $G_{i_0}$, it follows that $I$ intersects both $D_{i_0}$ and $\R^2\setminus \br{D_{i_0}}$. We let $w \in I\cap (\R^2\setminus \br{D_{i_0}})$. If $w\in \mathcal S$, then $(I\setminus \{z\})\cap \mathcal S\neq \emptyset$ and there is nothing to prove. If $w\in D_j$ for some $j\in \N$, $j\neq i_0$, we consider the segment $[z,w]\subset I$. Since $z\in \partial D_{i_0}$ and $D_j\cap \partial D_{i_0}=\emptyset$, it follows that $[z,w]$ is not contained in $D_{j}$. Therefore there exists a point $\widetilde w\in [z,w]\cap \partial D_{j}$; see Figure \ref{figure:nonconvex}. The line $L$ does not intersect the set $F$, so $\widetilde w \neq z$. It follows that $(I\setminus \{z\})\cap \mathcal S\neq \emptyset$ in this case too. The proof is complete.
\end{proof}

\bibliography{packings_dimension}
\end{document}